\documentclass[11pt]{amsart}  
\usepackage{amsthm, amsmath, amscd, amssymb, latexsym, stmaryrd, color}
\usepackage[top=2.5cm, bottom=2.5cm, left=2cm, right=2cm]{geometry}
\usepackage[colorinlistoftodos]{todonotes}
\theoremstyle{plain}

\newtheorem*{theorem*}{Theorem}
\newtheorem{theorem}{Theorem}[section]
\newtheorem{lem}[theorem]{Lemma}
\newtheorem{prop}[theorem]{Proposition}

\newtheorem{cor}[theorem]{Corollary}
\theoremstyle{definition}
\newtheorem{convention}[theorem]{Convention}
\newtheorem{defn}[theorem]{Definition}

\newtheorem{ex}[theorem]{Example}

\newtheorem{rmk}[theorem]{Remark}

\newtheorem*{conjecture}{Conjecture}
\theoremstyle{remark}

\usepackage[mathscr]{eucal}
\usepackage{ulem}
\usepackage{graphics, graphpap}
\usepackage{array, tabularx, longtable}
\usepackage{color}
\numberwithin{equation}{section}

\usepackage{url}
\usepackage[T1]{fontenc}
\usepackage{array,epsfig}
\usepackage{amsmath}
\usepackage{old-arrows}
\usepackage{amsfonts}
\usepackage{amssymb}
\usepackage{amsxtra}
\usepackage{latexsym}
\usepackage{dsfont}
\usepackage[thinlines]{easytable}
\usepackage{makecell}
\usepackage{tabularray}
\usepackage{mathrsfs}
\usepackage{color}
\usepackage[all]{xy}
\usepackage{}
\usepackage{xfrac}
\usepackage{tikz}
\usepackage{tikz-cd}

\makeatletter
\newcommand\mathcircled[1]{%
  \mathpalette\@mathcircled{#1}%
}
\newcommand\@mathcircled[2]{%
  \tikz[baseline=(math.base)] \node[draw,circle,inner sep=1pt,color=red] (math) {$\m@th#1#2$};%
}
\makeatother 

\usepackage[pagebackref=true]{hyperref}

\usepackage{graphicx}
\usepackage{mathtools}
\usepackage{caption}
\usepackage[OT2,T1]{fontenc}
\usepackage[page,toc,titletoc,title]{appendix}

\hypersetup{
 colorlinks,
 linkcolor=blue, citecolor=blue
 }

\newcommand{\et}{\textnormal{\'et}}
\newcommand{\nis}{\textnormal{Nis}}
\newcommand{\tame}{\textnormal{tame}}

\newcommand{\adic}{\textnormal{adic}}

\newcommand{\id}{\operatorname{id}}
\newcommand{\pr}{\operatorname{pr}}
\newcommand{\Hom}{\operatorname{Hom}}
\newcommand{\Sch}{\operatorname{Sch}}
\newcommand{\algspc}{\operatorname{AlgSpc}}
\newcommand{\affsch}{\operatorname{AffSch}}
\newcommand{\Sm}{\operatorname{Sm}}

\renewcommand{\sp}{\operatorname{sp}}
\newcommand{\da}{\mathbf{DA}^{\et}}

\newcommand{\deret}{\mathbf{D}^{\et}}
\newcommand{\sh}{\mathbf{SH}_{\mathfrak{M}}^{\tau}}
\newcommand{\shbb}{\mathbb{SH}_{\mathfrak{M}}^{\tau}}
\newcommand{\shbbct}{\mathbb{SH}_{\mathfrak{M},\mathrm{ct}}^{\tau}}
\newcommand{\shct}{\mathbf{SH}_{\mathfrak{M},\mathrm{ct}}^{\tau}}
\newcommand{\qush}{\mathbf{QUSH}_{\mathfrak{M}}^{\tau}}
\newcommand{\qushct}{\mathbf{QUSH}_{\mathfrak{M},\mathrm{ct}}^{\tau}}
\newcommand{\mnor}{\textnormal{M}}
\renewcommand{\th}{\operatorname{Th}}
\newcommand{\Fscr}{\mathscr{F}}
\newcommand{\Gscr}{\mathscr{G}}
\newcommand{\Rscr}{\mathscr{R}}
\newcommand{\Ex}{\operatorname{Ex}}

\newcommand{\gm}{\mathbb{G}}

\newcommand{\mscr}{\mathscr{M}}
\newcommand{\var}{\operatorname{Var}}

\newcommand{\preshv}{\mathbf{PreSh}}
\newcommand{\shv}{\mathbf{Sh}}

\newcommand{\sets}{\mathbf{Sets}}

\newcommand{\colim}{\operatorname{colim}}

\newcommand{\flag}{\mathcal{F}\ell}

\newcommand{\mfrak}{\mathfrak{M}}

\newcommand{\spect}{\mathbf{Spect}}
\newcommand{\Spec}{\operatorname{Spec}}
\DeclareSymbolFont{cyrletters}{OT2}{wncyr}{m}{n}
\DeclareMathSymbol{\Sha}{\mathalpha}{cyrletters}{"58}
\DeclareMathSymbol{\Be}{\mathalpha}{cyrletters}{"42}

\title[The integral identity conjecture in motivic homotopy theory]{The integral identity conjecture in motivic homotopy theory}  

\author[Khoa Bang. P]{Khoa Bang Pham}
\address{The Hong Kong University of Science and Technology\newline \indent Clear Water Bay, Kowloon, Hong Kong}
\email{phamkb@ust.hk}
\thanks{}


\begin{document}           
\begin{abstract}
   The integral identity conjecture of Kontsevich and Soibelman plays an important role in proving the existence of motivic Donaldson-Thomas invariants for three-dimensional noncommutative Calabi-Yau manifolds. There are a number of different formulations of this conjecture in different contexts, and accordingly, there are corresponding solutions to them. The methods devoted to solving this conjecture are diverse, ranging from $\ell$-adic cohomology of rigid analytic varieties to Hrushovski-Kazhdan motivic integration and motivic Fubini theorem for tropicalization maps,... In \cite{florian-2024}, Ivorra deduces a functorial version of the integral identity in the motivic stable homotopy categories of schemes, from the Braden hyperbolic localization theorem. This functorial version concerns Ayoub's nearby cycles functor associated with a $\gm_m$-equivariant function $f \colon \mathbb{V}(\mathcal{E}) \longrightarrow \mathbb{A}^1$ on a vector bundle $\mathbb{V}(\mathcal{E})$ over a field of characteristic zero. In the present work, we follow the functorial approach from \cite{florian-2024} and extend the scope of the original conjecture by Kontsevich and Soibelman by studying more generally the case of $\gm_m$-equivariant functions on algebraic $S$-spaces with a $\tau$-locally linearizable action of $\gm_m$ over a noetherian base scheme $S$.
\end{abstract}
\maketitle                 


\section*{Introduction}
\subsection*{State of the art}
In algebraic geometry, \textit{Donaldson-Thomas invariants} is an invariant named after Thomas and his advisor Donaldson. In the thesis \cite{thomas-2000}, Thomas introduces numbers called Donaldson-Thomas invariants associated with moduli spaces of coherent sheaves on compact Calabi-Yau threefold. These numbers can be defined as integrals of the cohomological class $1$ over the virtual fundamental class. In \cite{behrend-2009}, Behrend realized that Donaldson-Thomas invariants can be computed as integrals of some constructible functions, called \textit{Behrend functions}. By \cite{joyce+song-2012}, one knows that the moduli space of coherent sheaves can be locally represented as the critical locus of the Chern-Simons functional and in such a case the value of the Behrend function is (up to a sign) the Euler characteristic of the Milnor fiber of the corresponding Chern-Simons functional. In \cite{kontsevich+soibelman-2008}, by replacing Milnor fibers with motivic Milnor fibers (living in Grothendieck rings of varieties) in the sense of Denef-Loeser \cite{denef+loeser-1998}, generalized to formal schemes, Kontsevich and Soibelman study the \textit{motivic Donaldson-Thomas invariants} for compact Calabi-Yau threefolds, which are motivic refinements of numerical Donaldson-Thomas invariants. In \cite{kontsevich+soibelman-2008}, the authors know that motivic Donaldson-Thomas invariants are special values of an important homomorphism from the motivic Hall algebra to the motivic quantum torus ring. By using the motivic Thom-Sebastiani theorem \cite{denef+loeser-1999-2}\cite{looijenga-2002}, generalized to formal schemes, the multiplicative property of such homomorphism is equivalent to the \textit{motivic integral identity}, formulated as follows
\begin{conjecture}[Kontsevich-Soibelman's integral identity for regular functions]
Let $k$ be a field of characteristic zero. Let $X$ be a $k$-variety. Let $d_1,d_2$ be nonnegative integers, let $\gm_{m,k}$ acts on $\mathbb{A}_k^{d_1} \times_k \mathbb{A}_k^{d_2} \times_k X$ with positive weights on the first factor, negative weights on the second factor and trivially on $X$. Let $f \colon \mathbb{A}_k^{d_1} \times_k \mathbb{A}_k^{d_2} \times_k X \longrightarrow \mathbb{A}_k^1$ be a $\gm_{m,k}$-equivariant morphism, where $\mathbb{A}_k^1$ is endowed with the trivial $\gm_{m,k}$-action. Let $X$ be embedded in $\mathbb{A}_k^{d_1} \times_k \mathbb{A}_k^{d_2} \times_k X$ by zero sections and let $f_{\mid X}$ be the restriction of $f$ on $X$. The equality
\begin{equation*}
    \int_{\mathbb{A}^{d_1}_k}(\psi_f)_{\mid \mathbb{A}^{d_1}_k \times_k Y} = \mathbf{L}^{d_1}(\psi_{f\mid Y})
\end{equation*}
holds in $\mathscr{M}^{\hat{\mu}}_X$. 
\end{conjecture}
The original conjecture stated in \cite{kontsevich+soibelman-2008} is the case of the formal completion $X = \mathbb{A}_k^{d_3}$ and $f$ acts by weights $(1,-1,0)$. Let us review some processes of solving this conjecture in literature. In \cite{kontsevich+soibelman-2008} and \cite{thuong-2019}, the $\ell$-adic version with $X = \mathbb{A}_k^{d_3}$ and weights $(1,-1,0)$ is solved by Kontsevich, Soibelman and Lê. In \cite{thuong-2015}, using Hrushovski-Kazhdan motivic integration, Lê proved that the integral identity conjecture for algebraically closed fields of characteristic zero holds in some localization of $\mathscr{M}_k^{\hat{\mu}}$ (see also \cite{thuong-2012}). In \cite{nicaise-2019}, using motivic Fubini theorem for the tropicalization map, Nicaise and Payne remove this localization and improve the result to general $X$ and arbitrary weights over fields of characteristic zero containing all roots of unity. In \cite{le+nguyen-2020}, after developing the equivariant motivic integration, the authors also prove the conjecture for regular functions without additional hypothesis. Currently, Bu in \cite{bu-2024} shows that the conjecture is true for $(-1)$-shifted symplectic stacks (over $k$). These approaches have some common points: they are formulated in the virtual setting, namely, Grothendieck rings and more or less use motivic integration (over a field of characteristic zero) along with the tools it provides. Our purpose of this work is to obtain the functorial analogue of this conjecture in categories of motives (in contrast to virtual motives), after the recent work of Ivorra \cite{florian-2024}.

The theory of nearby cycles also exists in motivic homotopy theory thanks to the work  \cite{ayoub-thesis-2}. It is therefore reasonable to expect that the integral identity holds in the motivic world. Recently, in \cite{florian-2024}, Ivorra tackles the conjecture in this setting for vector bundles over $X$ with $X$ a varieties over a field $k$ of characteristic zero. The method used in Ivorra's work replies on the theory of spaces with $\gm_m$-actions. In comparison to those methods above, Ivorra's approach is functorial and holds more generally for any stable homotopical $2$-functor in the sense of \cite{ayoub-thesis-1}. Under Euler characteristics with compact support (see \cite{florian+julien-2013} or \cite{ayoub+florian+julien-2017}), Ivorra's result recovers the results above in the corresponding Grothendieck rings of variaties. Under good hypothesis, one can also eliminate the independence on characteristic and work with general base scheme; which, to our best knowledge, is not available in the virtual setting.

Let us say few words about the theory of spaces $X$ with $\gm_m$-actions. These are central objects in the study \cite{bialynicki+birula-1973}. Any such space "decomposes" into three parts: the space of fixed points $X^0$, the points $X^+$ floating to $X^0$, and the points $X^-$ floating away from $X^0$. In \cite{braden-2003}, Braden proved a theorem on localizing $\gm_m$-equivariant objects on $X$ to the subspace $X^0$. This theorem, called the \textit{Braden hyperbolic localization theorem}, is then quickly used in a number of places, especially in geometric representation theory and geometric Langlands. Subsequent studies after Braden's work can be found in \cite{drinfeld+gaitsgory-2014}\cite{drinfeld-2015} and \cite{richarz-2018}. In this work, we follow \cite{richarz-2018} mostly. Together with the formalism of operations of diagrams of algebraic spaces, we prove a functorial version of the integral identity for a variety of specialization system (not just nearby functors), of which nearby cycles functors is a special case. 

\subsection*{Main results} Let us state the following main result (see theorem \ref{compatible of Braden transformations with sp induced by derivators}) of this work, which can be viewed as a functorial generalization of the virtual integral identity and a motivic analogue of \cite[Theorem 3.3]{richarz-2018} (in fact, by taking \'etale realizations, one obtains \cite[Theorem 3.3]{richarz-2018})
\begin{theorem*} 
  Let $S$ be a noetherian scheme of finite dimension and $\sh$ the stable motivic homotopy category. Let $\Psi$ denote the nearby functors of Ayoub in \cite{ayoub-thesis-2}. Let $X$ be an algebraic $S$-space endowed with an $\tau$-locally linearizable $\gm_{m,S}$-action and, $X^0$ be its space of fixed points. Given a $\gm_{m,S}$-equivariant morphism $f \colon X \longrightarrow \mathbb{A}^1_S$ where the target is endowed with the trivial $\gm_{m,S}$-action, then there exists a commutative diagram
   \begin{equation*}
        \begin{tikzcd}[sep=large]
            (\pi^-_{\sigma})_*(e^-_{\sigma})^!\Psi_f(A) \arrow[d] & \Psi_{f^0}(\pi^-_{\eta})_*(e^-_{\eta})^!(A) \arrow[d] \arrow[l] \\ 
            (\pi^+_{\sigma})_!(e^+_{\sigma})^*\Psi_f(A) \arrow[r]  & \Psi_{f^0} (\pi_{\eta}^+)_!(e^+_{\eta})^*(A)
        \end{tikzcd}
    \end{equation*}
   natural in $A \in \sh(X_{\eta})$ with $f^0$ the restriction of $f$ to $X^0$. Moreover, all arrows are isomorphisms provided that $A$ is equivariant.
\end{theorem*}
We stress to the fact that the existence of the arrows in the theorem above is non-trivial and in general cannot be deduced from the given base change morphisms of nearby functors (see \cite[Definition 3.1.1]{ayoub-thesis-2}) unless in some special cases like vector bundles considered in \cite{florian-2024}. Their existences are consequences of the geometry of fixed points, attractors and repellers together with the content of the appendix together. We also note that in \cite{cass+hove+scholbach-2024}, the authors obtained a similar result for unipotent nearby functors, though they use different tools from ours. The theorem above implies the following (for details, see theorem \ref{categorical integral identity}, corollary \ref{monodromic categorical integral identity} and corollary \ref{the virtual integral identity})
\begin{theorem*}
  Let $S$ be a noetherian scheme of finite dimension. Let $Y$ be an algebraic $S$-space endowed with the trivial $\gm_{m,S}$-action.  Let $d_1,d_2$ be nonnegative integers, let $\gm_{m,S}$ acts on $X = \mathbb{A}_S^{d_1} \times_S \mathbb{A}_S^{d_2} \times_S Y$ with positive weights on the first factor, negative weights on the second factor and trivially on $Y$. Let $f \colon \mathbb{A}_S^{d_1} \times_S \mathbb{A}_S^{d_2} \times_S Y \longrightarrow \mathbb{A}_S^1$ be a $\gm_{m,S}$-equivariant morphism, where $\mathbb{A}_S^1$ is endowed with the trivial $\gm_{m,S}$-action. Then there is an isomorphism
   \begin{equation*}
       \int_{\mathbb{A}_S^{d_1}} \big(\Psi_f(\mathds{1}_{X_{\eta}}) \big)_{\mid \mathbb{A}^{d_1}_S \times_k X} \simeq \mathds{1}_{Y_{\sigma}}(-d_1)[-2d_1] \otimes \Psi_{f_{\mid Y}}(\mathds{1}_{Y_{\eta}}),
  \end{equation*}
in $\sh(Y_{\sigma})$. In particular, if $S = \Spec(k)$ with $k$ a field of characteristic zero, then after \cite{ayoub+florian+julien-2017}\cite{florian+julien-2013}, there is an identity
\begin{equation*}
    \int_{\mathbb{A}^{d_1}_k}(\psi_f)_{\mid \mathbb{A}^{d_1}_k \times_k Y} = \mathbf{L}^{d_1}(\psi_{f\mid Y})
\end{equation*}
in the ring $K_0(\sh(Y))$.
\end{theorem*}
This theorem is an extension of \cite[Corollary 3.2]{florian-2024} to the case where $Y$ is an algebraic space and not only a scheme. We note that in a recent work \cite{bu-2024}, Bu proves that the integral identity is true for stacks with affine stabilizers (in particular, algebraic spaces) using methods that are closely related to the ones used in this work. In other words, our work can be seen as a functorial version of what are done in \cite{bu-2024} but we emphasize on the fact that our isomorphism holds not only in characteristic zero (as done in \cite{florian-2024} and \cite{bu-2024}) but also in positive characteristic, where virtual nearby cycles do not make sense. 

\subsection*{Another result}

Let us mention another result which is of independent interest. In the \'etale setting, the authors \cite{richarz-2021} show that there exists a functorial analogue of the compatibility of the Bernstein isomorphism with the constant term map in \cite{haines-2014}. Here we prove a motivic version of such a result that realizes to the one in \cite[Theorem A]{richarz-2021} via the $\ell$-adic realization. 
\begin{theorem*}
Let $F$ be a non-archimedean local field with ring of integers $R$. Let $G$ be a connected reductive $F$-group with a choice of Levi $M$. Let $\mathcal{G}$ be a parahoric $R$-group scheme with generic fiber $G$. Given a cocharacter $\lambda \colon \gm_{m,R} \longrightarrow \mathcal{G}$ whose centralizer $\mathcal{M}$ is a parahoric $R$-group with generic fiber $M$, then there exists a natural transformation
 \begin{equation*}
     \operatorname{CT}_{\mathcal{M}} \circ \Psi^{\mathrm{tot}}_{\mathcal{G}} \longrightarrow \Psi^{\mathrm{tot}}_{\mathcal{M}} \circ \operatorname{CT}_M
 \end{equation*}
of functors $\da(\operatorname{Gr}_G,\mathbb{Q}) \longrightarrow \da(\flag_{\mathcal{M},\overline{\sigma}},\mathbb{Q})$, where $\operatorname{CT}_?$ denotes constant term functors and $\Psi^{\mathrm{tot}}_?$ denotes total nearby functors, $\operatorname{Gr}_?$ and $\flag_?$ are affine Grassmannians and flag varieties, respectively. The transformation is an isomorphism on $\gm_{m,F}$-equivariant objects, where the conjugation action is induced by $\lambda$. 
\end{theorem*}
\subsection*{Structure of the manuscript}
In section {\color{blue}\S 1}, we remind the reader on various specialization systems that are of our interest in this paper. In section {\color{blue}\S 2}, we study diagrams of algebraic spaces with $\gm_m$-actions and the commutation of three operations $((f,\alpha)_{\#},(f,\alpha)^*,(f,\alpha)_*)$ with Braden transformations at the level of diagrams of algebraic spaces (in fact, we can do this with $(f_!,f^!)$ but we do not need this so we omit it). We then state the Braden hyperbolic localization theorem for diagrams of algebraic spaces and prove it by making various reductions to the case of schemes. In the same section, by taking \'etale realization, we recover results in \'etale setting. In section {\color{blue} \S 3}, we would like to show the commutation of Braden transformations with specialization systems. In order to do this, we use the commutation of five operations with Braden transformtions in the previous section. The commutation is applied to deduce the desired functorial motivic integral identity in section {{\color{blue}\S 4}; in the same section, we define quasi-unipotent motives of algebraic spaces and demonstrate that our non-virtual results realize to the virtual ones as in \cite{florian-2024}\cite{bu-2024}. In section {\color{blue}\S 5}, we extend specialization systems to ind-schemes and deduce the commutation of nearby functors on affine Grassmannians and flag varieties with constant term functors.

Along with the main theorems, we develop a four-functor formalism for algebraic spaces in the appendix {\color{blue}\S 6} with a view towards algebraic derivators and specialization systems. This formalism is implicitly used throughout the work. More concretely, we generalize results in \cite{chowdhury-2024}\cite{chowdhury+angelo-2024} and \cite{khan+ravi-2024} for $\mathbf{SH}$ to general $\sh$ (including $\da$, the category of \'etale motives). We study operations at the level of algebraic derivators (in the sense of \cite{ayoub-thesis-2}) in order to have the language to develop the theory of specialization systems for algebraic spaces in section {\color{blue}\S 1} as well as the commutation of Braden transformations with operations. In order to make the formalism quickly accessible and compatible with \cite{ayoub-thesis-2}, we use the traditional approach of stable model categories instead of the approach using stable $\infty$-categories. However, given the equivalence between presentable $\infty$-categories and combinatorial, simplicial model categories, two approaches are formally the same. 

\subsection*{Acknowledgement} This paper is a part of the PhD project of the author when he was a PhD student at University of Rennes under the supervision of Florian Ivorra and Julien Sebag. The author would like to express the deepest gratitude to his supervisors for their vision, support and patience throughout his years as a PhD student. The author also thanks Joseph Ayoub for answering his questions on algebraic derivators during the preparation of this work. Furthermore, the author thanks the anonymous referee for his careful reading and valuable comments. Finally, this paper was revised while the author was a postdoc in the group of Quoc Ho at the Hong Kong University of Science and Technology, supported by the Hong Kong RGC GRF grants 16304923 and 16301324.
\section*{Conventions}

\subsection*{Spaces} For a scheme $S$, we denote by $\mathcal{O}_S$ its structural sheaf. An $S$-variety is a separated $S$-scheme of finite type, a morphism of $S$-varieties is a morphism of $S$-schemes.  By an algebraic $S$-space $X$, we mean a sheaf
\begin{equation*}
    X \colon (\Sch_S)^{\mathrm{op}}_{\text{fppf}} \longrightarrow \mathbf{Sets}
\end{equation*}
on the fppf-site with representable diagonal and admitting a surjective \'etale morphism $U \longrightarrow X$ with $U$ being a scheme. In this article, we assume all algebraic spaces are noetherian and separated over $S$. In particular, quasi-separatedness shows that there exists a Nisnevich atlas $V \longrightarrow X$ with $V$ an (affine) scheme (see \cite[Theorem 6.4]{knutson-1971}).

\subsection*{Ind-schemes} Let $\operatorname{AffSch}$ denote the category of affine schemes. A \textit{(strict) ind-scheme} is a presheaf $X \colon \affsch^{\mathrm{op}} \longrightarrow \sets$ which admits a presentation $X \simeq \colim_{i \in I}X_i$ as a filtered colimits of schemes $X_i$ with closed transitions $\iota_{i \to j} \colon X_i \longrightarrow X_j$. 

If \textbf{P} is a property of schemes, then an ind-scheme $X \simeq \colim_{i \in I}X_i$ is said to be ind-\textbf{P} if each $X_i$ is \textbf{P}. Similarly, if \textbf{P} is a property of morphisms of schemes, then a morphism $X \longrightarrow Y$ of ind-schemes is said to be ind-\textbf{P} if it is a colimit of $X_i \longrightarrow Y_i$, each one is \textbf{P}. We say that a schematic morphism $f \colon X \longrightarrow Y$ is \textit{cartesian} if there exists a presentation $X \simeq \colim_{i \in I}X_i$ so that $Y \simeq \colim_{i \in I} (X_i \times_X Y)$ is a presentation of $Y$. 
 
\subsection*{Categories and base change} We assume that the reader is familiar with base change notations in \cite{ayoub-thesis-1}, such as $\Ex_*^*,\Ex_!^*,$,... which have obvious meanings that the reader can guess easily. 

Throughout the article, we denote by $\sh$ the motivic stable homotopy categories of Morel-Voevodsky-Ayoub. In \cite{ayoub-thesis-1}\cite{ayoub-thesis-2}, $\sh$ is also extended to diagrams of schemes. In order to distinguish two cases, we write $\sh(X)$ for the case of schemes and $\shbb(\Fscr,I)$ for the case of diagrams of schemes. We refer to the appendix for more details and generalizations. 
\section{Specialization systems} 
The theory of specialization systems is developed in \cite{ayoub-thesis-2} as a generalization of the well-known nearby cycles functors. The definitions given in \cite{ayoub-thesis-2} extend directly to algebraic spaces with slight modifications. This section provides a short introduction to specialization systems as well as some typical examples, including motivic nearby functors. Given a diagram of algebraic $S$-spaces
\begin{equation*}
    \begin{tikzcd}[sep=large]
        \eta \arrow[r,"j"] & B & \sigma \arrow[l,"i",swap] 
    \end{tikzcd}
\end{equation*}
and for a morphism $f \colon X \longrightarrow B$ of algebraic $S$-spaces, we form the following diagram by base change 
\begin{equation*}
    \begin{tikzcd}[sep=large]
        X_{\eta} \arrow[r,"j_f"]  \arrow[d,"f_{\eta}",swap] & X  \arrow[d,"f"] & X_{\sigma} \arrow[d,"f_{\sigma}"] \arrow[l,"i_f",swap] \\ 
        \eta \arrow[r,"j"] & B & \sigma . \arrow[l,"i",swap]
    \end{tikzcd}
    \end{equation*} 
We make the following definition. 
\begin{defn} 
A \textit{specialization system} over $(B,\eta,\sigma)$ (or over $(B,j,i)$ if there is no confusion) is a collection of triangulated functors 
\begin{equation*} 
\sp_f \colon \sh(X_{\eta}) \longrightarrow \sh(X_{\sigma})
\end{equation*} 
associated with morphisms $f \colon X \longrightarrow B$ so that for another morphism $g \colon Y \longrightarrow X$, there are a natural transformation $\alpha_g$ and its adjoint $\beta_g$
\begin{equation*}
      \begin{tikzcd}[sep=large]
         \sh(X_{\eta}) \arrow[r,"\sp_f"] \arrow[d,"(g_{\eta})^*",swap] & \sh(X_{\sigma}) \arrow[Rightarrow, shorten >=27pt, shorten <=27pt, dl,"\alpha_g"] \arrow[d,"(g_{\sigma})^*"] \\ 
         \sh(Y_{\eta}) \arrow[r,"\sp_{f \circ g}"] &  \sh(Y_{\sigma}). 
    \end{tikzcd} \ \ \ \  \begin{tikzcd}[sep=large]
         \sh(Y_{\eta}) \arrow[r,"\sp_{f \circ g}"] \arrow[d,"(g_{\eta})_*",swap] & \sh(Y_{\sigma}) \arrow[d,"(g_{\sigma})_*"] \\ 
         \sh(X_{\eta}) \arrow[Rightarrow, shorten >=27pt, shorten <=27pt, ur,"\beta_g"] \arrow[r,"\sp_f"] &  \sh(X_{\sigma}).
         \end{tikzcd} 
\end{equation*}
We require $\alpha_g,\beta_g$ to be isomorphisms when $g$ is smooth, proper, respectively. The triplet $(B,\eta,\sigma)$ is called a \textit{base} when we work with specialization systems.
\end{defn} 
In \cite{bang-2024} (and \cite{ayoub-thesis-2} under the assumption on quasi-projectiveness of schemes), there are base change morphisms
\begin{equation*}
\begin{tikzcd}[sep=large]
         \sh(Y_{\eta}) \arrow[r,"\sp_{f \circ g}"] \arrow[d,"(g_{\eta})_!",swap] & \sh(Y_{\sigma}) \arrow[Rightarrow, shorten >=27pt, shorten <=27pt, dl,"\mu_g"] \arrow[d,"(g_{\sigma})_!"] \\ 
         \sh(X_{\eta}) \arrow[r,"\sp_f"] &  \sh(X_{\sigma})
    \end{tikzcd} \ \ \ \ \begin{tikzcd}[sep=large]
         \sh(X_{\eta}) \arrow[d,"(g_{\eta})^!",swap]\arrow[r,"\sp_f"] & \sh(X_{\sigma}) \arrow[d,"(g_{\sigma})^!"]   \\ 
         \sh(Y_{\eta})   \arrow[r,"\sp_{f \circ g}"] \arrow[Rightarrow, shorten >=27pt, shorten <=27pt, ur,"\nu_g"] &  \sh(Y_{\sigma}) 
    \end{tikzcd} 
    \end{equation*} 
defined by means of the Nagata compactification theorem. The effort in \cite{bang-2024} is devoted to prove that $\mu_g$ is well-defined (independent of the choice of the compactification). The formal proof showed in \cite{bang-2024} extends directly to the case of algebraic spaces so we may freely use them like when we work with schemes. Here are some examples that are of our interest. For any diagram $(\Fscr,I) \longrightarrow \eta$ of $\eta$-schemes and a specialization system $\sp$ over $(B,j,i)$, we form the diagram consisting of cartesian squares
    \begin{equation*}
    \begin{tikzcd}[sep=large]
       (\Fscr_X,I) \arrow[r,"\pi_f^{\Fscr}"] \arrow[d,"f_{\eta}^{\Fscr}",swap] & (X_{\eta},I) \arrow[r,"j_f"] \arrow[d,"f_{\eta}"] & (X,I) \arrow[d,"f"] & (X_{\sigma},I) \arrow[l,"i_f",swap] \arrow[d,"f_{\sigma}"] \arrow[r,"p_I"] & X_{\sigma} \arrow[d,"f_{\eta}"] \\ 
         (\Fscr,I) \arrow[r,"\pi_{\id}^{\Fscr}"] & (\eta,I)\arrow[r,"j"] & (B,I) & (\sigma,I) \arrow[r,"p_I"] \arrow[l,"i",swap]  & \sigma
    \end{tikzcd}
\end{equation*}
and define a specialization system $(\Fscr \bullet \chi)$ by the formula
\begin{equation*}
    (\Fscr \bullet \chi)_f \coloneqq (p_I)_{\#}\sp_f(\pi_f^{\Fscr})_*(\pi_f^{\Fscr})^*(p_I)^*
\end{equation*}
(it is a specialization system by \cite[Lemme 3.2.7, 3.2.8]{ayoub-thesis-2} and \cite{bang-2024}). Let us mention some concrete examples. The case $(B,\eta,\sigma) = (\mathbb{A}_S^1,\gm_{m,S},S)$ with $i$ the zero section and $j$ the canonical inclusion. Let $\Delta$ be the category of finite ordinals $\mathbf{n} = \left \{0 < 1 < \cdots < n \right \}$ and  $\mathbb{N}^{\times}$ be the set of non-zero natural numbers viewed as a category whose objects are non-zero natural numbers and morphisms are defined via the opposite of the division relation. In \cite[Definition 3.5.1]{ayoub-thesis-2}, Ayoub defines a diagram of $\gm_{m,S}$-schemes $\theta_{\id} \colon(\mathscr{R},\Delta \times \mathbb{N}^{\times}) \longrightarrow (\gm_{m,S},\Delta \times \mathbb{N}^{\times})$ such that $\mathscr{R}(\mathbf{n},r) = \gm_{m,S} \times_S (\gm_{m,S})^n$. The structural morphism is given by the composition
    \begin{equation*}
    \gm_{m,S} \times_S (\gm_{m,S})^n \overset{\pr_1}{\longrightarrow} \gm_{m,S} \overset{e_r}{\longrightarrow} \gm_{m,S}, 
    \end{equation*}
    in which $\pr_1$ is the projection on the first factor and the second one is the $r$-power morphism. Let $I \subset \Delta \times \mathbb{N}^{\times}$ be a subcategory, we consider the cartesian diagram 
    \begin{equation*}
        \begin{tikzcd}[sep=large]
           (\mathscr{R}^I,I) \arrow[r] \arrow[d] & (\gm_{m,S},I)\arrow[d] \\ 
           (\mathscr{R},\Delta \times \mathbb{N}^{\times}) \arrow[r] & (\gm_{m,S},\Delta \times \mathbb{N}^{\times}).
           \end{tikzcd} 
    \end{equation*}
By making suitble choices of $I$, we obtain corresponding specialization systems over $(\mathbb{A}_S^1,\gm_{m,S},S)$ that we are interested in. 

\begin{ex} \label{ex: standard specialization system} 
There is an obvious choice with $\sp_f = i_f^*(j_f)_*$ and $(\Fscr,I) = \eta$ is the identity and the resulting specialization system is called the \textit{standard specialization system}, denoted by $\chi_f$ in \cite{ayoub-thesis-2}.
\end{ex}
In the examples below, we always take $\sp_f$ to be $\chi_f$.
\begin{ex} \label{ex: unipotent nearby functors}
Let $I = \Delta \times \left \{1 \right \} \simeq \Delta$, the resulting specialization system is the unipotent nearby functors $\Upsilon_f$. Under appropriate hypotheses, the unipotent nearby functors are isomorphic to the log specialization system $\operatorname{log}$. For details, one can consult \cite{ayoub-thesis-2}\cite{ayoub-realisation-etale}. 
\end{ex}

\begin{ex} \label{tame nearby functors}
Let $I = \Delta \times \mathbb{N}^{'\times}$, with $\mathbb{N}^{'\times}$ the full subcategory of $\mathbb{N}^{\times}$ consisting of numbers invertible on $S$, the resulting specialization system is the \textit{tame nearby functors} $\Psi^{\tame}_f$ (in \cite{ayoub-realisation-etale}, it is called \textit{moderate nearby cycles functors}). We note that if $S$ is a scheme over $\mathbb{Q}$, then $\mathbb{N}^{'\times} = \mathbb{N}$ and in this case $\Psi^{\tame}_f$ is called \textit{motivic nearby functors} in \cite{ayoub-thesis-2}.
\end{ex}

\begin{ex} \label{adic nearby functors}
 The case $(B,\eta,\sigma) = (S,\eta,\sigma)$ with $S$ the spectrum of an excellent strictly henselian discrete valuation ring, $\eta$ its generic point and $\sigma$ its closed point. Let $\pi$ be a fixed uniformizer. The uniformizer $\pi$ induces a morphism $S \longrightarrow \mathbb{A}_S^1$. For morphisms $f \colon X \longrightarrow S$, the specialization systems $\Psi^{\adic}_f \coloneqq \Psi^{\tame}_{\pi \circ f}$ is called the \textit{$\pi$-adic nearby functors} (if $\pi$ is clear, we may just speak of adic nearby functors). Two notable cases are $S = \Spec(\mathbb{Z}_p)$ and $\sh = \da$ or $S = \Spec(k[\![t]\!])$ (with $k$ algebraically closed of characteristic zero) and $\sh$ can be either $\mathbf{SH}$ or $\da$. For details, one can consult \cite{ayoub-rigid-motive}\cite{ayoub+florian+julien-2017}.
\end{ex}

\begin{ex} \label{total nearby functors}
Let $K$ be a non-archimedean local field with ring of integers $R$ and finite residue field $k$ of characteristic $p$ and cardinality $q$, i.e., either $K/\mathbb{Q}_p$ is a finite extension or $K \simeq \mathbb{F}_q(\!(t)\!)$ is a local function field. We fix a uniformizer $\pi$ of the maximal ideal of $R$. For any morphism $f \colon X \longrightarrow S$ (with $S=\Spec(R)$). Fix a separable closure $K^{\mathrm{sep}}/K$ and let $K^{\mathrm{ur}}$ the maximal unramified extension of $K$ in $K^{\mathrm{sep}}$. Let $\overline{S} = \Spec(\overline{R})$ where $\overline{R}$ is the integral closure of $\mathcal{O}$ in $K^{\mathrm{sep}}$. Let $\overline{\sigma} = \Spec(\overline{k})$ where $\overline{k}$ is the residue field of $\overline{R}$, the integral closure of $R$. We let $K^{\mathrm{tr},p}$ be the union of finite extensions of $K^{\mathrm{ur}}$ of degree a power of $p$. With these data, we define the total nearby functor (following \cite[Definition 10.14]{ayoub-realisation-etale}) to be 
\begin{align*}
    \Psi_X^{\mathrm{tot}} \colon \da(X_{\eta},\Lambda) & \longrightarrow \da(X_{\overline{\sigma}},\Lambda) \\ 
    A & \longmapsto \operatorname{HoColim} (\Psi_{X \times_S T}^{\tame}(A_{\mid X_{\eta} \otimes_K L}))_{\mid X_{\overline{\sigma}}}
\end{align*}
where $\iota \colon \Spec(\overline{k}) \longrightarrow \Spec(k)$ is induced by the canonical inclusion. This is a specialization system over the base $(S,\eta,\overline{\sigma})$. For the reader would prefer a formula of $\Psi^{\mathrm{tot}}_X$ in terms of algebraic derivators, he can consult \cite{preis-2023}. 
\end{ex}

\section{The Braden hyperbolic localization theorem}

In this section, we extend the scope of the Braden hyperbolic localization theorem to diagrams of algebraic spaces. The applications we have in mind are those diagrams appearing in the definitions of specialization systems reminded in the preceding section. 
\subsection{Generalities on equivariance} We begin with extending the notion of $\gm_m$-equivariant objects (see for instance \cite{richarz-2018}) to the level of derivators and study their behaviors under several operations. 
\begin{defn}
    Let $I$ be a small category and $(\Fscr,I)$ be a diagram of $S$-varieties. A $\gm_{m,S}$-\textit{action} of $(\Fscr,I)$ is a cartesian morphism 
    \begin{equation*}
        a \colon \gm_{m,S} \times_S (\Fscr,I) \longrightarrow (\Fscr,I).
    \end{equation*}
Let $I,J$ be two small categories, let $(\Fscr,I)$ and $(\Gscr,I)$ be two diagrams equipped with $\gm_{m,S}$-action. A morphism $(f,\alpha) \colon (\Gscr,J) \longrightarrow (\Fscr,I)$ is said to be $\gm_{m,S}$-\textit{equivariant} if the diagram
    \begin{equation*}
        \begin{tikzcd}[sep=large]
            \gm_{m,S} \times_S (\Gscr,J) \arrow[r,"(\id{,}(f{,}\alpha))"] \arrow[d,"a",swap] &  \gm_{m,S} \times_S (\Fscr,I)  \arrow[d,"a"] \\ 
            (\Gscr,J) \arrow[r,"(f {,}\alpha) "] & (\Fscr,I)
        \end{tikzcd}
    \end{equation*}
    is cartesian. Since we only consider $\gm_{m,S}$-actions in this article, it is safe to just say equivariant morphisms instead of $\gm_{m,S}$-equivariant morphisms. 
\end{defn}

\begin{defn}
    Let $(\Fscr,I)$ be a diagram of $S$-varieties equipped with a $\gm_{m,S}$-action, we denote by 
    \begin{align*}
        a \colon \gm_{m,S} \times_S (\Fscr,I) & \longrightarrow (\Fscr,I) \\ 
        p \colon \gm_{m,S} \times_S (\Fscr,I) & \longrightarrow (\Fscr,I)
    \end{align*}
    the action and the projection, respectively. An object $A \in \shbb(\Fscr,I)$ is called \textit{equivariant} if $p^*(A)$ is isomorphic to $a^*(A)$ in $\shbb(\gm_{m,S} \times_S (\Fscr,I))$.
\end{defn} 

\begin{lem} \label{equivariance of six operations}
    Let $I,J$ be small categories and $(\Gscr,J),(\Fscr,I)$ be diagrams of $S$-varieties endowed with $\gm_{m,S}$-actions. Let $(f,\alpha) \colon (\Gscr,J) \longrightarrow (\Fscr,I)$ be an equivariant morphism between diagrams of $S$-varieties endowed with $\gm_{m,S}$-actions. 
    \begin{enumerate} 
    \item The operations $((f,\alpha)^*,(f,\alpha)_*)$ preserve equivariant objects. In particular, the operations $(i^*,i_*)$ preserve equivariant objects for any $i \in I$. 
    \item If $(f,\alpha)$ is smooth objectwise, then $(f,\alpha)_{\#}$ preserves equivariant objects. 
    \item If $\alpha = \id$ and $f$ is compactifiable, then $f_!$ preserves equivariant objects.
    \item If $I = J = \mathbf{e}$, i.e., $f$ is a morphism of schemes, then $f^!$ preserves equivariant objects.
    \end{enumerate}
\end{lem}

\begin{proof}
The proof is formally identical to \cite[Lemma 2.5]{richarz-2018} with slight modifications because we are working with diagrams. The case $(f,\alpha)^*$ follows from functoriality. The case $(f,\alpha)_*$ follows from \cite[Corollaire 2.4.21]{ayoub-thesis-1} (note that $a,p$ are cartesian). The case $(f,\alpha)_{\#}$ (with $(f,\alpha)$ smooth objectwise) follows from \cite[Corollaire 2.4.24]{ayoub-thesis-2} \footnote{We note that in \cite[Corollaire 2.4.24]{ayoub-thesis-2}, there is a typo, the correct base change morphism should be $(f',\alpha)_{\#}(g')^* \longrightarrow g^*(f,\alpha)_{\#}$.}. Finally, the case $f_!$ ($\alpha = \id$ and $f$ compactifiable) follows from proposition \ref{proper + smooth base change theorem for diagrams}.
\end{proof}
We deduce two corollaries on specialization systems. Let $(B,\eta,\sigma)$ be a base consisting of algebraic $S$-spaces endowed with the trivial $\gm_{m,S}$-action. 
\begin{cor}
  Let $X$ be an algebraic $S$-space endowed with a $\gm_{m,S}$-action. Let $f \colon X \longrightarrow B$ be an equivariant morphism. Given a diagram of algebraic $(\Fscr,I) \longrightarrow \eta$ or an object $A \in \sh(\eta)$, then specialization systems 
  \begin{equation*}
      (\Fscr \bullet \chi)_f\colon \sh(X_{\eta}) \longrightarrow \sh(X_{\sigma})
  \end{equation*}
  send equivariant objects to equivariant objects.  
\end{cor}

\begin{proof}
We endow $(\Fscr,I)$ with the trivial action, making the structural morphism $(\Fscr,I) \longrightarrow \eta$ equivariant. The equivariances become obvious since these specialization systems can be expressed in terms of functors preserving equivariant objects. 
\end{proof}
\subsection{The Braden transformation and the Braden hyperbolic localization theorem}
In this section, we define the so-called Braden transformation and study its interaction to four operations at the level of diagrams of varieties. Most of the results here are direct adaptations of section 3.1 in \cite{richarz-2018}. We make a convention first.
\begin{convention}
The notations $\mathbb{A}_S^1,(\mathbb{A}_S^1)^+,(\mathbb{A}_S^1)^-$ denote the same underlying scheme $\mathbb{A}_S^1$ but with different $\gm_{m,S}$-actions; respectively, they are the trivial action, the multiplication $(\lambda,x) \longmapsto \lambda x$ and its inverse $(\lambda,x) \longmapsto \lambda^{-1}x$. The schemes $\gm_{m,S},S$ are always assumed to be endowed with the trivial actions.
\end{convention}
\begin{defn} 
Let $X$ be an algebraic $S$-space with a $\gm_{m,S}$-action. Follow \cite{drinfeld+gaitsgory-2014}\cite{richarz-2018}, we can define three new spaces. If $T$ is an $S$-scheme, then we define the \textit{space of fixed points as}, the \textit{attractor}, the \textit{repeller}, respectively
\begin{align*}
X^0(T) & = \left \{\text{equivariant morphisms} \ T \longrightarrow X_T \right \} \\ 
    X^+(T) & = \left \{\text{equivariant morphisms} \ (\mathbb{A}^1_T)^+ \longrightarrow X_T\right \} \\ 
    X^-(T) & = \left \{\text{equivariant morphisms} \ (\mathbb{A}^1_T)^- \longrightarrow X_T\right \}.
\end{align*}
The structural morphisms $(\mathbb{A}^1_S)^{\pm} \longrightarrow S$ define morphisms $s^{\pm} \colon X^0 \longrightarrow X^{\pm}$. The zero sections $S \longrightarrow (\mathbb{A}_S^1)^{\pm}$ define morphisms $\pi^{\pm} \colon X^{\pm}  \longrightarrow X^0$ such that $\pi^{\pm} \circ s^{\pm} = \id$. Analogously, the inclusions $\gm_{m,S} \longhookrightarrow (\mathbb{A}_S^1)^{\pm}$ define morphisms $e^{\pm} \colon X^{\pm} \longrightarrow X$ such that $e^+ \circ s^+ = e^- \circ s^-$ is the inclusion functor $X^0 \longhookrightarrow X$. The commutative diagram
   \begin{equation*}
        \begin{tikzcd}[sep=large]
        X^0 \arrow[rrd,bend left = 15,"s^-"] \arrow[ddr,bend right = 15,"s^+",swap] \arrow[dr,"u"] & & &  \\ 
          &   X^+ \times_X X^-\arrow[r,"t^-"] \arrow[d,"t^+",swap] & X^{-} \arrow[d,"e^-"] \arrow[r,"\pi^-"] &  X^0\\ 
          &  X^+ \arrow[d,"\pi^+",swap] \arrow[r,"e^+"] & X  & \\ 
           & X^0  & &
        \end{tikzcd}
    \end{equation*}
 is called the \textit{hyperbolic localization diagram}, denoted $\mathrm{HypLoc}(X)$. By \cite[Proposition 1.17]{richarz-2018}, $s^+,s^-$ are closed immersions and $u=(s^+,s^-)$ is an open and closed immersion. 
\end{defn} 

\begin{defn}
For a diagram $(\Fscr,I)$ endowed with a $\gm_{m,S}$-action, we can define the \textit{space of fixed points} $(\Fscr^0,I)$, the \textit{attractor} $(\Fscr^+,I)$, the \textit{repeller} $(\Fscr^-,I)$ of $(\Fscr,I)$ objectwise, i.e., for $i$ in $I$
 \begin{equation*}
     \Fscr^0(i)  \coloneqq (\Fscr(i))^0 \ \ \ \ \Fscr^+(i)  \coloneqq (\Fscr(i))^+ \ \ \ \ \Fscr^-(i)  \coloneqq (\Fscr(i))^-.
 \end{equation*}
 We say that the \textit{hyperbolic localization diagram} of $(\Fscr,I)$ exists if there is a commutative diagram $\mathrm{HypLoc}(\Fscr,I)$ of the form
  \begin{equation*} \label{eq: hyperbolic localization diagram}
    \begin{tikzcd}[sep=large]
        (\Fscr^0,I) \arrow[rrd,bend left = 15,"s^-"] \arrow[ddr,bend right = 15,"s^+",swap] \arrow[dr,"u"] & & & \\ 
          &   (\Fscr^+,I) \times_{(\Fscr,I)} (\Fscr^-,I) \arrow[r,"t^-"] \arrow[d,"t^+",swap] \arrow[dr,phantom] & (\Fscr^{-},I) \arrow[d,"e^-"] \arrow[r,"\pi^-"] &  (\Fscr^0,I) \\ 
          &  (\Fscr^+,I) \arrow[d,"\pi^+",swap] \arrow[r,"e^+"] & (\Fscr,I) & \\ 
           &  (\Fscr^0,I)  & &
        \end{tikzcd}
    \end{equation*}
   where $\mathrm{HypLoc}(\Fscr,I)(i) \coloneqq \mathrm{HypLoc}(\Fscr(i))$ and all of the involved morphisms are strongly compactifiable. 
\end{defn}

\begin{defn} \label{Braden transformation}
With the notation above, the transformation 
\begin{align*}
      L^-_{(\Fscr,I)} =   (\pi^-)_*(e^-)^!  & \longrightarrow (\pi^-)_*(s^-)_*(s^-)^*(e^-)^! \overset{\sim}{\longrightarrow} u^*(t^-)^*(e^-)^! \\ 
     & \overset{\Ex^{!*}}{\longrightarrow} u^*(t^+)^!(e^+)^*  \overset{\sim}{\longrightarrow} u^!(t^+)^!(e^+)^* \\  
      & \overset{\sim}{\longrightarrow} (s^+)^!(e^+)^* \overset{\sim}{\longrightarrow} (\pi^+)_!(s^+)_!(s^+)^!(e^+)^*  \longrightarrow L^+_{(\Fscr,I)} = (\pi^+)_!(e^+)^* 
\end{align*}
is called a \textit{Braden transformation} and we denote by $\phi_{(\Fscr,I)} \colon  (\pi^-)_*(e^-)^! \longrightarrow (\pi^+)_!(e^+)^*$ this transformation.
\end{defn} 

The Braden theorem (see for instance \cite{braden-2003}\cite{richarz-2018} or \cite{florian-2024}) states that Braden transformations are isomorphisms on equivariant objects for nice spaces.
\begin{defn}
    Let $X$ be an algebraic $S$-space. A $\gm_{m,S}$-action on $X$ is said to be \textit{$\tau$-locally linearizable} if there exists a $\gm_{m,S}$-equivariant $\tau$-covering $\left \{U_i \longrightarrow X \right \}_{i \in I}$, where $U_i$ are $S$-affine schemes with $\gm_{m,S}$-actions.
\end{defn}

\begin{ex}  \label{linearizable of schemes}
Let us denote by Zar and \'et to indicate the Zariski and the \'etale topology, respectively.
\begin{enumerate} 
    \item Any affine $S$-variety is Zar-locally linearizable. 
    \item Let $k$ be an algebraically closed field and $U$ be a normal $k$-variety endowed with a $\gm_{m,k}$, then Sumihiro theorem (see \cite{sumihiro-1974}\cite{sumihiro-1975}) asserts that $U$ is Zar-locally linearizable. 
    \item As noted in \cite{richarz-2018}, in the work \cite[Theorem 10.1 and Corollary 10.2]{alper2023etale} (see also \cite[Theorem 2.4]{alper2021luna} for the case of $S$ being the spectrum of an algebraically closed field), it is shown that if the structural morphism $X \longrightarrow S$ is locally of finite presentation and $X$ is quasi-separated then $X$ is \'et-locally linearizable.
\end{enumerate}
\end{ex}

We state the generalized version of Braden's hyperbolic localization theorem. 
\begin{theorem} \label{braden theorem}
    Let $S$ be a quasi compact and quasi-separated scheme. Let $(\Fscr,I)$ be a diagram of algebraic $S$-spaces equipped with a $\gm_{m,S}$-action so that $\mathrm{HypLoc}(\Fscr,I)$ exists. Suppose that $\shbb$ is $\tau$-conservative. Then for any equivariant object $A$ in $\shbb(\Fscr,I)$, the Braden transformation
    \begin{equation*}
      (\pi^-)_*(e^-)^!(A) \longrightarrow (\pi^+)_!(e^+)^*(A)
    \end{equation*}
    is an isomorphism in the following cases: 
    \begin{itemize}
        \item ($\tau$-Sch)$(\Fscr,I) = X$ is an algebraic $S$-space endowed with a $\tau$-locally linearizable $\gm_{m,S}$-action.
        \item (AffDia) $(\Fscr,I)$ consists of affine $S$-varieties and in this case, we can drop the $\tau$-conservativity.
    \end{itemize}
\end{theorem} 

We skip the proof to upcoming sections.
\subsection{The interaction of Braden transformations with operations}

The following theorem is a (necessary for the later purposes) generalization of \cite[Proposition 3.1]{richarz-2018} to the level of diagrams of algebraic spaces. The proof is almost identical with some modifications. We propose here a proof using the purity exchange structure $\Ex^{!*}$ and moreover, since we are working with diagrams of spaces, there are details needed to be stressed in the proof. 
 
\begin{theorem} \label{compatible of Braden transformations with four operations}
    Let $I,J$ be small categories and $(f,\alpha) \colon (\Gscr,J) \longrightarrow (\Fscr,I)$ be a morphism of diagrams of $S$-varieties equipped with $\gm_{m,S}$-actions. Suppose that hyperbolic localization diagrams of $(\Fscr,I)$ and $(\Gscr,J)$ exist and 
    \begin{equation*}
        \mathrm{HypLoc}(\Gscr,J) = \mathrm{HypLoc}(\Fscr,I) \times_{(\Fscr,I)} (\Gscr,J).
    \end{equation*}
 Let $(f^0,\alpha) \colon (\Gscr^0,J) \longrightarrow (\Fscr^0,I)$ be the induced morphism on the diagrams of spaces of fixed points. 
    \begin{enumerate}
        \item There is a natural $2$-morphism
         \begin{equation*}
       \begin{tikzcd}[sep=large]
          L^-_{(\Fscr,I)}(f,\alpha)_* \arrow[d,"\phi_{(\Fscr{,}I)}",swap] \arrow[r,"(\Phi^-)_*"] & (f^0,\alpha)_*L^-_{(\Gscr,J)} \arrow[d,"\phi_{(\Gscr{,}J)}"] \\ 
           L^+_{(\Fscr,I)}(f,\alpha)_* \arrow[r,"(\Phi^-)_*"] & (f^0,\alpha)_*L^+_{(\Gscr,J)}.
       \end{tikzcd}
    \end{equation*}
        as a $2$-morphism of functors $\shbb(\Gscr,J) \longrightarrow \shbb(\Fscr^0,I)$. Moreover, if $\alpha = \id$ and $F = (f,\alpha)$ is proper objectwise, then $(\Phi^-)_*,(\Phi^+)_*$ are isomorphisms.
        \item There is a natural $2$-morphism 
         \begin{equation*}
       \begin{tikzcd}[sep=large]
           (f^0,\alpha)^*L^-_{(\Fscr,I)} \arrow[d,"\phi_{(\Fscr{,}I)}",swap] \arrow[r,"(\Phi^-)^*"] & L^-_{(\Gscr,J)}(f,\alpha)^* \arrow[d,"\phi_{(\Gscr{,}J)}"] \\ 
           (f^0,\alpha)^*L^+_{(\Fscr,I)} \arrow[r,"(\Phi^-)^*"] & L^+_{(\Gscr,J)}(f,\alpha)^*.
       \end{tikzcd}
   \end{equation*}
        as a $2$-morphism of functors $\shbb(\Fscr,I) \longrightarrow \shbb(\Gscr^0,J)$. Moreover, if $(f,\alpha)$ is smooth objectwise and $e^{\pm}$ are cartesian, then $(\Phi^-)^*,(\Phi^+)^*$ are isomorphisms. 
        \item Let $X$ be an algebraic $S$-space endowed with a $\gm_{m,S}$-action and $p_I \colon (X,I) \longrightarrow X$ be the projection, then there is a natural $2$-morphism
         \begin{equation*}
       \begin{tikzcd}[sep=large]
           (p_I)_{\#}L^-_{(X,I)}  \arrow[d] \arrow[r]  &   L^-_X(p_I)_{\#} \arrow[d] \\ 
           (p_I)_{\#}L^+_{(X,I)} \arrow[r] & L^+_X(p_I)_{\#}.
        \end{tikzcd}
         \end{equation*}
         as a $2$-morphism of functors $\shbb(X,I) \longrightarrow \sh(X^0)$. \footnote{This holds more generally under some technical assumptions but the case $p_I \colon (X,I) \longrightarrow X$ is the only case that we use in this work.}
         \end{enumerate}
\end{theorem}

\begin{rmk}
    We can define furthermore the base change morphisms with respect to $f_!,f^!$ (here we require $\alpha=\id$ and $f$ compactifiable). As we do not need the morphisms $\Phi_!$ and $\Phi^!$ so we refer \cite{richarz-2018} to those readers who are interested in these morphisms.
\end{rmk}

\begin{proof}
In this proof, by "by adjunction", we mean an application of \cite[Corollaire 1.1.14]{ayoub-thesis-1}. The hypothesis implies that there are cartesian squares
   \begin{equation*}
       \begin{tikzcd}[sep=large]
           (\Gscr^0,J)\arrow[d,"(f^0{,}\alpha)",swap] \arrow[dr,phantom,"(\textnormal{C}^{\pm}_1)"]  & (\Gscr^{\pm},J) \arrow[dr,phantom,"(\textnormal{C}^{\pm}_2)"] \arrow[l,"\pi^{\pm}", swap] \arrow[r,"e^{\pm}"] \arrow[d,"(f^{\pm}{,}\alpha)"] & (\Gscr,J) \arrow[d,"(f{,}\alpha)"] \\ 
           (\Fscr^0,I)  & (\Fscr^{\pm},I) \arrow[l,"\pi^{\pm}", swap] \arrow[r,"e^{\pm}"] & (\Fscr,I) 
       \end{tikzcd}
   \end{equation*}
 The morphisms $(\Phi^-)^*,(\Phi^+)^*$ are defined by
 \begin{equation*}
       \begin{tikzcd}[sep=large]
           (f^0,\alpha)^*(\pi^-)_*(e^-)^!  \arrow[d] \arrow[r,"\Ex^*_*(\textnormal{C}^-_1)"]  & (\pi^-)_*(f^-,\alpha)^*(e^-)^! \arrow[r,"\Ex^{!*}(\textnormal{C}^-_2)"] &  (\pi^-)_*(e^-)^!(f,\alpha)^* \arrow[d] \\ 
           (f^0,\alpha)^*(\pi^+)_!(e^+)^* \arrow[r,"\Ex^*_!(\textnormal{C}^+_1)"] & (\pi^+)_!(f^+,\alpha)^*(e^+)^* \arrow[r,equal] & (\pi^+)_!(e^+)^*(f,\alpha)^*.
       \end{tikzcd}
   \end{equation*}
  If $(f,\alpha)$ is smooth objectwise and $e^-$ is cartesian, then from the explicit description above, we see that both $(\Phi^-)^*$ and $(\Phi^+)^*$ are isomorphisms by \cite[Lemme 2.4.26]{ayoub-thesis-1}. In such the case, by adjunction, we can consider the transformation $\Phi_{\#}=((\Phi^-)_{\#},(\Phi^+)_{\#})$ given by 
\begin{equation*}
      (f^0,\alpha)_{\#}\phi_{(\Gscr,J)} \longrightarrow (f^0,\alpha)_{\#}\phi_{(\Gscr,J)}(f,\alpha)^*(f,\alpha)_{\#} \overset{(\Phi^*)^{-1}}{\longrightarrow} (f^0,\alpha)_{\#}(f^0,\alpha)^*\phi_{(\Fscr,I)}(f,\alpha)_{\#} \longrightarrow  \phi_{(\Fscr,I)}(f,\alpha)_{\#}.
  \end{equation*}
  or in full diagram, 
   \begin{equation*}
       \begin{tikzcd}[sep=large]
           (f^0,\alpha)_{\#}(\pi^-)_*(e^-)^!  \arrow[d] \arrow[r,"\Ex_{\# *}(\square_1^-)"]  & (\pi^-)_*(f^-,\alpha)_{\#}(e^-)^! \arrow[r,"\Ex^{!}_{\#}(\square_2^-)"] &  (\pi^-)_*(e^-)^!(f,\alpha)_{\#} \arrow[d] \\ 
           (f^0,\alpha)_{\#}(\pi^+)_!(e^+)^* \arrow[r,"\Ex_{\#!}(\square_2^+)"] & (\pi^+)_!(f^+,\alpha)_{\#}(e^+)^* \arrow[r,"\Ex_{\#}^*(\square_1^+)"] &  (\pi^+)_!(e^+)^*(f,\alpha)_{\#}
       \end{tikzcd}
   \end{equation*}
  and the case $p_I \colon (X,I) \longrightarrow X$ is of course a special case. In general, one cannot hope that $e^-$ is cartesian (see \cite[Example 1.13]{richarz-2018}); however we still claim that it makes sense to write $\Ex_{\#}^!$ in the case $p_I \colon (X,I) \longrightarrow X$. Indeed, one can define 
  \begin{equation*}
      (p_I)_{\#}(e^-)^! \longrightarrow (e^-)^!(e^-)_!(p_I)_{\#}(e^-)^! \overset{(\Ex_{\#!})^{-1}}{\longrightarrow} (e^-)^!(p_I)_{\#}(e^-)_!(e^-)^! \longrightarrow (e^-)^!(p_I)_{\#}
  \end{equation*}
 thanks to proposition \ref{direct images commute with homotopy colimits}. By adjunction again, two morphisms $(\Phi^-)_*,(\Phi^+)_*$ are given by 
\begin{equation*}
    \phi_{(\Fscr,I)}(f,\alpha)_* \longrightarrow (f^0,\alpha)_*(f^0,\alpha)^* \phi_{(\Fscr,I)}(f,\alpha)*  \overset{\Phi^*}{\longrightarrow} (f^0,\alpha)\phi_{(\Gscr,J)}(f,\alpha)^*(f,\alpha)_* \longrightarrow (f^0,\alpha)\phi_{(\Gscr,J)}.
\end{equation*}
More precisely, 
   \begin{equation*}
       \begin{tikzcd}[sep=large]
           (\pi^-)_*(e^-)^! (f,\alpha)_* \arrow[d] \arrow[r,"(\Ex^!_*(\textnormal{C}^-_2))^{-1}"]  & (\pi^-)_*(f^-,\alpha)_*(e^-)^! \arrow[r,equal] &  (f^0,\alpha)_*(\pi^-)_*(e^-)^! \arrow[d] \\ 
           (\pi^+)_!(e^+)^*(f,\alpha)_* \arrow[r,"\Ex^*_*(\textnormal{C}^+_2)"] & (\pi^+)_!(f^+,\alpha)_*(e^+)^* \arrow[r,"\Ex_{!*}(\textnormal{C}^+_1)"] &  (f^0,\alpha)_*(\pi^+)_!(e^+)^*.
       \end{tikzcd}
   \end{equation*}
The rest of the proof is formally the same as \cite[Proposition 3.1]{richarz-2018}.
\end{proof}

We give a criterion to produce a lot of pairs of diagrams satisfying the hypothesis in the theorem above. We need a lemma first.

\begin{lem} \label{descent of fixed points, attractors, repellers}
    Let $X,Y$ be algebraic $S$-spaces endowed with $\gm_{m,S}$-actions. Let $f \colon X \longrightarrow Y$ be an equivariant \'etale morphism. Assume that there exists a cartesian square
    \begin{equation*}
        \begin{tikzcd}[sep=large]
                  U \arrow[r] \arrow[d] & V \arrow[d] \\ 
                  X \arrow[r] & Y
        \end{tikzcd}
    \end{equation*}
    with $U,V$ being affine $S$-varieties and $U \longrightarrow X, V \longrightarrow Y$ being $\tau$-coverings, then 
    \begin{align*}
        X^0  & =  X \times_Y Y^0 \\ 
        X^{\pm} & = X^0 \times_{Y^0} Y^{\pm}
    \end{align*}
    as functors.
\end{lem}
\begin{proof}
    We prove the case functors of fixed points, the attractors and repellers are proved similarly. By \cite[Theorem 1.8]{richarz-2018}, $U^0 \longrightarrow X^0$ and $V^0 \longrightarrow Y^0$ are \'etale coverings. By \'etale descent, $X^0 =  X \times_Y Y^0$ if and only if $X^0 \times_{X \times_Y Y^0} X \times_Y V^0 = X \times_Y V^0$. Thanks to \cite[Lemma 1.10]{richarz-2018}, one has $V^0 = V \times_Y Y^0$ and hence
    \begin{equation*}
        X^0 \times_{X \times_Y Y^0}( X \times_Y V^0) =  X^0 \times_{X \times_Y Y^0} (X \times_Y V \times_Y Y^0) = X^0 \times_Y V.
    \end{equation*}
    Again, by applying \cite[Lemma 1.10]{richarz-2018} repeatedly, one obtains cartesian squares (where cartesian of the upper left square follows from the fact that $U^0 \longrightarrow U$ is a closed immersion)
    \begin{equation*}
        \begin{tikzcd}[sep=large] 
        U^0 \arrow[r,equal] \arrow[d,equal] &  U^0 \arrow[d] \arrow[r] & V^0 \arrow[d] \\
        U^0 \arrow[r] \arrow[d]  &  U \arrow[r] \arrow[d] & V \arrow[d] \\ 
            X^0 \arrow[r] &     X \arrow[r] & Y
        \end{tikzcd}
    \end{equation*}
    which shows that $X^0 \times_Y V = X \times_Y V^0 = U^0$ as desired. 
\end{proof}
\begin{prop}  \label{hyperbolic localization diagram of all specialization systems considered}
    Let $S$ be a quasi-compact quasi-separated scheme. Let $I \subset \Delta \times \mathbb{N}^{'\times}$ be a subcategory with $\mathbb{N}^{'\times}$ the set of invertible natural numbers on $S$. Let $\mathscr{R}^{\tame,I}$ be the diagram obtained by base change
    \begin{equation*}
        \begin{tikzcd}[sep=large]
           (\mathscr{R}^{\tame,I},I) \arrow[r] \arrow[d] & (\gm_{m,S},I)\arrow[d] \\ 
           (\mathscr{R}^{\tame},\Delta \times \mathbb{N}^{'\times}) \arrow[r] & (\gm_{m,S},\Delta \times \mathbb{N}^{'\times}).
        \end{tikzcd}
    \end{equation*}
    Let $X$ be an algebraic $S$-space endowed with a $\tau$-locally linearizable $\gm_{m,S}$-action. Let $f \colon X \longrightarrow \mathbb{A}^1_S$ be a $\gm_{m,S}$-equivariant morphism. The hyperbolic localiztion diagram of $(\mathscr{R}_X^{\tame,I},I)$ exists and 
    \begin{equation*}
        \mathrm{HypLoc}(\mathscr{R}_X^{\tame,I},I) = \mathrm{HypLoc}(X_{\eta},I) \times_{(X_{\eta},I)} (\mathscr{R}^{\tame,I},I).
    \end{equation*}
\end{prop}

\begin{proof}
    For any $r \in \mathbb{N}^{'\times}$, there is a cartesian square
    \begin{equation*}
        \begin{tikzcd}[sep=large]
          X^r_{\eta} \arrow[d] \arrow[r,"e_r"] & X_{\eta} \arrow[d] \\ 
          \gm_{m,S} \arrow[r,"e_r"] & \gm_{m,S}.
        \end{tikzcd}
    \end{equation*}
    Since $r$ is invertible on $S$, the morphism $e_r \colon \gm_{m,S} \longrightarrow \gm_{m,S}$ is \'etale. Endow $\gm_{m,S}$ with the trivial action and $X^r_{\eta}$ with the diagonal action, the morphism $e_r\colon X^r_{\eta} \longrightarrow X_{\eta}$ is $\gm_{m,S}$-equivariant and \'etale. By lemma \ref{descent of fixed points, attractors, repellers}, we have natural identifications
    \begin{align*}
        (X^r_{\eta})^0 & = (X_{\eta})^0  \times_{X_{\eta}} X^r_{\eta} \\ 
        (X^r_{\eta})^{\pm} & = (X_{\eta})^{\pm}  \times_{X_{\eta}}  X^r_{\eta}  .
    \end{align*}
    At the level of objects $(\mathbf{n},r)$ in $I \subset \Delta \times \mathbb{N}^{'\times}$, we claim that
     \begin{equation*}
        \mathrm{HypLoc}(\mathscr{R}_X(\mathbf{n},r)) = \mathrm{HypLoc}(X_{\eta}) \times_{X_{\eta}} \mathscr{R}_X(\mathbf{n},r).
    \end{equation*}
    Indeed, we have
    \begin{equation*}
         (\mathscr{R}_X(\mathbf{n},r))^0 =  (X_{\eta}^r)^0 \times_S  ((\gm_{m,S})^n)^0 = (X_{\eta})^0  \times_{X_{\eta}} X^r_{\eta} \times_S (\gm_{m,S})^n =  (X_{\eta})^0 \times_{X_{\eta}} \mathscr{R}_X(\mathbf{n},r),
    \end{equation*}
    since $\gm_{m,S}$ acts trivially on $(\gm_{m,S})^n$. The cases of attractors and repellers are proved similarly
    \begin{equation*}
         (\mathscr{R}_X(\mathbf{n},r))^{\pm} = (X_{\eta}^r)^{\pm} \times_S  ((\gm_{m,S})^n)^{\pm} = (X_{\eta})^{\pm}  \times_{X_{\eta}}  X^r_{\eta} \times_S (\gm_{m,S})^n  = (X_{\eta})^{\pm} \times_S  \mathscr{R}_X(\mathbf{n},r),
    \end{equation*}
    which follows from the fact that morphisms $\mathbb{A}_S^1 \longrightarrow \gm_{m,S}$ are constant. To show that $\mathrm{HypLoc}(\mathscr{R}_X)$ exists at the level of morphisms, we note that any morphism in $\Delta \times \mathbb{N}^{'\times}$ just affects the second factor in $\mathrm{HypLoc}(X_{\eta}) \times_S (\gm_{m,S})^n$. Finally, to show that 
      \begin{equation*}
         \mathrm{HypLoc}(\mathscr{R}_X^{\tame,J},J) = \mathrm{HypLoc}(X_{\eta},J) \times_{(X_{\eta},J)} (\mathscr{R}^{\tame,J},J),
    \end{equation*}
    it is sufficient to prove that each square in the diagram
       \begin{equation*}
       \begin{tikzcd}[sep=large]
           ((\mathscr{R}_X^{\tame,I})^0,I)\arrow[d,"(\theta_f)^0",swap]  & ((\mathscr{R}_X^{\tame,I})^{\pm},I)  \arrow[l,"\pi^{\pm}", swap] \arrow[r,"e^{\pm}"] \arrow[d,"(\theta_f)^{\pm}"] & (\mathscr{R}_X^{\tame,I},I) \arrow[d,"\theta_f"] \\ 
           (X_{\eta}^0,I)  & (X_{\eta}^{\pm},I) \arrow[l,"\pi^{\pm}", swap] \arrow[r,"e^{\pm}"] & (X_{\eta},I) 
       \end{tikzcd}
   \end{equation*}
   is cartesian. This is easy since objectwise, vertical arrows are compositions of projections with power morphisms. 
\end{proof}

\subsection{The proof of the Braden hyperbolic localization theorem for diagram of algebraic spaces}

In this subsection, we prove theorem \ref{braden theorem}. We divide the proof into a series of lemmas, which eventually ends up with the case of schemes. 

\begin{lem}
   In theorem \ref{braden theorem}, (AffDia) holds true if and only if it holds true for affine $S$-varieties. 
\end{lem}

\begin{proof}
    By theorem \ref{compatible of Braden transformations with four operations}, for any $i \in I$, there exists a commutative diagram 
    \begin{equation*}
        \begin{tikzcd}[sep=large]
            i^*(\pi^-)_*(e^-)^!A \arrow[r,"(\Phi^-)^*"] \arrow[d,"\phi_{(\Fscr{,}I)}",swap] & (\pi^-(i))_*(e^-(i))^!i^*A \arrow[d,"\phi_{\Fscr(i)}"]  \\ 
           i^*(\pi^+)_!(e^+)^*A  \arrow[r,"(\Phi^+)^*"] & (\pi^+(i))_!(e^+(i))^*i^*A
        \end{tikzcd}
    \end{equation*}
    where the morphism $(\Phi^-)^*$ is given by
    \begin{equation*}
        i^*(\pi^-)_*(e^-)^! \overset{\Ex^*_*}{\longrightarrow} (\pi^-(i))_*i^*(e^-)^! \overset{\Ex^{*!}}{\longrightarrow} 
        (\pi^-(i))_*(e^-(i))^!i^*A,
    \end{equation*}
    hence an isomorphism by proposition \ref{proper + smooth base change theorem} and \cite[Lemme 2.4.26]{ayoub-thesis-1}, and the morphism $(\Phi^+)$ is given 
    \begin{equation*}
        i^*(\pi^+)_!(e^+)^* \overset{\Ex^*_!}{\longrightarrow} (\pi^+(i))_!i^*(e^+)^*  = (\pi^+(i))_!(e^+(i))^*i^*
    \end{equation*}
    hence an isomorphism by proposition \ref{proper + smooth base change theorem}. The lemma then follows. 
\end{proof}

\begin{prop}
  In theorem \ref{braden theorem}, (AffDia) holds true for vector bundles of finite rank. 
\end{prop}

\begin{proof}
    See \cite{richarz-2018} or \cite{florian-2024}.
\end{proof}

\begin{lem} \label{compatible of Braden transformations with closed immersions}
    Let $Z,X$ be affine $S$-varieties endowed with $\gm_{m,S}$-actions. Let $z \colon Z \longhookrightarrow X$ be an equivariant closed immersion. There is a commutative diagram 
    \begin{equation*}
       \begin{tikzcd}[sep=large]
           (\pi^-)_*(e^-)^!z_* \arrow[d,"\phi_X",swap] \arrow[r,"(\Phi^-)_*"] & (z^0)_*(\pi^-)_*(e^-)^! \arrow[d,"\phi_Z"] \\ 
           (\pi^+)_!(e^+)^*z_* \arrow[r,"(\Phi^+)_*"] & (z^0)_*(\pi^+)_!(e^+)^*
       \end{tikzcd}
    \end{equation*}
    and the two vertical arrows are isomorphisms. 
\end{lem}

\begin{proof}
    The proof is identical to the proof of \cite[Lemma 2.22]{richarz-2018}.
\end{proof}

\begin{lem} \label{compatible of Braden transformations with etale morphisms}
    Let $U,X$ be $S$-varieties endowed with $\gm_{m,S}$-actions with $U$ being $S$-affine. Let $u \colon U \longrightarrow X$ be an equivariant \'etale morphism. There is a commutative diagram 
    \begin{equation*}
       \begin{tikzcd}[sep=large]
           (u^0)^*(\pi^-)_*(e^-)^!  \arrow[d,"\phi_X",swap] \arrow[r,"(\Phi^-)^*"] & (\pi^-)_*(e^-)^!u^* \arrow[d,"\phi_U"] \\ 
           (u^0)^*(\pi^+)_!(e^+)^* \arrow[r,"(\Phi^+)^*"] & (\pi^+)_!(e^+)^*u^*.
       \end{tikzcd}
   \end{equation*}
    and the two vertical arrows are isomorphisms.
\end{lem}

\begin{proof}
     The proof is identical to the proof of \cite[Lemma 2.24]{richarz-2018}.
\end{proof}

\begin{lem} \label{braden theorem for affdia}
    Theorem \ref{braden theorem} holds true for affine $S$-varieties with $\gm_{m,S}$-action.
\end{lem}

\begin{proof}
Suppose that $X$ is affine over $S$. By \cite[Lemma 2.21]{richarz-2018}, there exists Zariski locally on $S$ a $\gm_{m,S}$-equivariant closed immersion $X \longhookrightarrow \mathbb{A}^n_S$. Let $\left \{S_i \longrightarrow S \right \}$ be a Zariski covering such that each $X \times_S S_i \longrightarrow S_i$ factors as $X \times_S S_i \longhookrightarrow \mathbb{A}_{S_i}^n \longrightarrow S_i$ (all are equivariant morphisms). Denote by $\left \{u_i \colon X \times_S S_i \longrightarrow X \right \}$ the corresponding base change, then lemma \ref{compatible of Braden transformations with etale morphisms} shows that there exist commutative squares
    \begin{equation*}
        \begin{tikzcd}[sep=large]
             (u^0_i)^*(\pi^-)_*(e^-)^!  \arrow[d,"\phi_X",swap] \arrow[r,"(\Phi^-)^*"] & (\pi^-)_*(e^-)^!(u_i)^* \arrow[d,"\phi_{X \times_S S_i}"] \\ 
           (u^0_i)^*(\pi^+)_!(e^+)^* \arrow[r,"(\Phi^+)^*"] & (\pi^+)_!(e^+)^*(u_i)^*.
        \end{tikzcd}
    \end{equation*}
    Since a Zariski covering is also a Nisnevich covering and any $\sh$ is Nis-conservative (see \cite[Proposition 1.4.3]{ayoub-thesis-1}), we can reduce to the case of a $\gm_{m,S}$-equivariant closed immersion $v \colon X \longhookrightarrow \mathbb{A}_S^n$ (with $X$ being $S$-affine).  Since $(i^0)_*$ is fully faithful, we can reduce to proving that that $(v^0)_*(\pi^-)_*(e^-)^!  \longrightarrow (v^0)_*(\pi^+)_!(e^+)^*$ is an isomorphism. By lemma \ref{compatible of Braden transformations with closed immersions}, there is a transformation
  \begin{equation*}
       \begin{tikzcd}[sep=large]
           (\pi^-)_*(e^-)^! i_* \arrow[d,"\phi_{\mathbb{A}_S^n}",swap] \arrow[r,"(\Phi^-)_*"] & (i^0)_*(\pi^-)_*(e^-)^! \arrow[d,"\phi_{X}"] \\ 
           (\pi^+)_!(e^+)^*i_* \arrow[r,"(\Phi^+)_*"] & (i^0)_*(\pi^+)_!(e^+)^*
       \end{tikzcd}
    \end{equation*}
    with $(\Phi^-)_*,(\Phi^+)_*$ being isomorphisms. The morphism $\phi_{\mathbb{A}^n_S}$ is an isomorphism by Florian's theorem above and this forces $\phi_X$ to be an isomorphism. 
\end{proof}

\begin{prop}
   Theorem \ref{braden theorem} holds true for an algebirac $S$-space $X$ endowed with a $\tau$-locally linearizable $\gm_{m,S}$-action.
\end{prop}

\begin{proof}
By the property of being $\tau$-locally linearizable, there exists a $\tau$-covering $\left \{u_i \colon U_i \longrightarrow X \right \}$ consisting of equivariant morphisms, where each $U_i$ are affine $S$-schemes with $\gm_{m,S}$-actions. By $\tau$-conservativity in proposition \ref{descent of motives}, the transformation 
\begin{equation*}
    (\pi^-)_*(e^-)^! \longrightarrow (\pi^+)_!(e^+)^*
\end{equation*}
is an isomorphism if and only if
\begin{equation*}
    (u_i^0)^*(\pi^-)_*(e^-)^! \longrightarrow (u_i^0)^*(\pi^+)_!(e^+)^*
\end{equation*}
are isomorphisms. By theorem \ref{compatible of Braden transformations with four operations}, there are transformations
   \begin{equation*}
       \begin{tikzcd}[sep=large]
           (u^0_i)^*(\pi^-)_*(e^-)^!  \arrow[d,"\phi_X",swap] \arrow[r,"(\Phi^-)^*"] & (\pi^-)_*(e^-)^!(u_i)^* \arrow[d,"\phi_{U_i}"] \\ 
           (u^0_i)^*(\pi^+)_!(e^+)^* \arrow[r,"(\Phi^+)^*"] & (\pi^+)_!(e^+)^*(u_i)^*.
       \end{tikzcd}
   \end{equation*}
 We remark that since $u_i$ are \'etale, so are $u^0_i,u_i^{\pm}$ by \cite[Theorem 1.8]{richarz-2018}. As a consequence, $(u^0_{\pm})^* = (u^0_{\pm})^!$ and we still have that $\Phi^*$ is an isomorphism. This helps us to reduce to the case of $X$ being an affine $S$-scheme, which is known to be true thanks to the previous lemma.
\end{proof}
We need the following two lemmas in the next section. 
 \begin{lem}[Braden contraction lemma] \label{braden contraction lemma}
    Let $X$ be an algebraic $S$-space on which $\gm_{m,S}$ acts. Assume that the action is $\tau$-locally linearizable, then the transformations
     \begin{equation*}
     \begin{split} 
         (\pi^-_X)_* \longrightarrow (\pi^-_X)_*(s^-_X)_*(s^-_X)^* = (s_X^-)^* \\
         (s_X^+)^!  = (\pi_X^+)_!(s^+_X)_!(s_X^+)^! \longrightarrow (\pi_X^+)_!
         \end{split} 
     \end{equation*}
     are isomorphisms on equivariant objects in $\sh(X^0)$. Likewise, the unit $\id \longrightarrow (\pi^-)_*(\pi^-)^*$ is an isomorphism on equivariant objects. 
 \end{lem}
 
 \begin{proof}
 The two transformations are dual so let us prove the first one only. Let $\left \{u_i \colon U_i \longrightarrow X \right \}_{i \in I}$ be a $\tau$-covering consisting of equivariant morphisms with $U_i$ being $S$-affine. For any $i \in I$, there exists a natural commutative diagram
 \begin{equation*}
     \begin{tikzcd}[sep=large]
         (u^0_i)^*(\pi^-_X)_*  \arrow[r] \arrow[d] & (u^0_i)^*(s^-_X)^* \arrow[d] \\ 
         (\pi^-_U)_*(u^-_i)^* \arrow[r] & (s^-_U)^*(u^-_i)^*.
     \end{tikzcd}
 \end{equation*}
 by \cite[Lemma 2.25]{richarz-2018}. This reduces the problem to the case of schemes. A similar argument reduces the problem to the case $X$ is a vector bundle, which is known to be true by \cite[Lemma 2.19]{richarz-2018}. Reasoning in the same way, we only have to show that $\id \longrightarrow (\pi^-)_*(\pi^-)^*$ is an isomorphism if $X$ is a vector bundle, but this is the $\mathbb{A}^1$-homotopy property. 
 \end{proof}
 
 \begin{lem}[Braden purity theorem] \label{braden purity theorem}
    Let $X$ be an algebraic $S$-space on which $\gm_{m,S}$ acts. Assume that the action is $\tau$-locally linearizable, then the transformation 
    \begin{equation*}
        (s^-)^*(e^-)^!(A) \longrightarrow (s^+)^!(e^+)^*(A)
    \end{equation*}
    is an isomorphism on equivariant objects $A$ in $\sh(X)$. 
 \end{lem}
 
 \begin{proof}
 This is a consequence of the Braden contraction lemma and the Braden hyperbolic localization theorem.
 \end{proof}
 
 \subsection{Realization functors}
In \cite{ayoub-realisation-etale}, Ayoub constructs an equivalence between $\da(X,\Lambda)$ and the derived category of \'etale sheaves of $\Lambda$-modules  $\deret(X,\Lambda)$ under the assumption that $X$ is an $S$-variety with $S$ either a field or a Dedekind domain. The result extends naturally to algebraic spaces. 
\begin{lem} \label{etale realization}
Let $S$ be a scheme and $\Lambda$ be a $\mathbb{Z}/N\mathbb{Z}$-algebra for some integer $N$ invertible on $S$. Assume that the statement of \cite[Théorème 4.1]{ayoub-realisation-etale} is satisfied. Let $X$ be an algebraic $S$-space, then the functor $\hat{\iota}^*_X \colon \deret(X,\Lambda) \longrightarrow \da(X,\Lambda)$ is an equivalence of categories.  
\end{lem}
\begin{proof}
Let $u \colon U \longrightarrow X$ be an \'etale atlas of $X$. For any $A \in \da(X,\Lambda)$ and $B \in \deret(X,\Lambda)$, we have to show that 
\begin{align*}
    (\hat{\iota}_X)^*(\hat{\iota}_X)_*(A) \longrightarrow A \ \ \ \ \text{and} \ \ \ \ B \longrightarrow (\hat{\iota}_X)_*(\hat{\iota}_X)^*(B)
\end{align*}
are isomorphisms. For instance, apply $u^*$, we have to show that $u^*(B)\longrightarrow u^*(\hat{\iota}_X)_*(\hat{\iota}_X)^*(B)$ is an isomorphism. Thanks to (more precisely, an analogue of) \cite[Lemme 4.12]{ayoub-realisation-etale}, there exists a natural isomorphism $u^*(\hat{\iota}_X)_* \simeq (\hat{\iota}_U)_*u^*$ (because $u$ is \'etale), but this reduces to the case of schemes. 
\end{proof}
 Under the equivalence $\hat{\iota}^*$, six operations in the \'etale setting are translated to six operations in the motivic setting. In particular, one has
 \begin{theorem}
 Let $S$ be scheme and $\Lambda$ be a $\mathbb{Z}/N\mathbb{Z}$-algebra for some integer $N$ invertible on $S$. Assume that the statement of \cite[Théorème 4.1]{ayoub-realisation-etale} is satisfied, then the Braden theorem is satisfied with respect to $\deret(-,\Lambda)$.
 \end{theorem}

\section{Commutation of Braden transformations with various specialization systems}

Let $X$ be an algebraic $S$-space endowed with a $\gm_{m,S}$ action and $f \colon \mathbb{A}_S^1$ be a $\gm_{m,S}$-equivariant morphism. We let $I \subset \Delta \times \mathbb{N}^{'\times}$ be a subcategory and consider the diagram $\mathscr{R}^{\tame,I}$ considered in proposition \ref{hyperbolic localization diagram of all specialization systems considered}. We denote by $f^0 \colon X^0 \longrightarrow \mathbb{A}_S^1$ the restriction of $f$ to the space of fixed points. By an abuse of notation, we still denote by $\theta_{\id} \colon (\Rscr^{\tame,I},I) \longhookrightarrow (\gm_{m,S},I)$ the structural morphism. 
\begin{prop} \label{compatible of Braden transformations with sp induced by derivators}
Let $X$ be an algebraic $S$-space endowed with a $\tau$-locally linearizable $\gm_{m,S}$-action. Given an equivariant morphism $f \colon X \longrightarrow \mathbb{A}_S^1$, then we have a commutative square
       \begin{equation*}
        \begin{tikzcd}[sep=large]
            (\pi^-_{\sigma})_*(e^-_{\sigma})^!(\Rscr^{\tame,I} \bullet \chi)_f \arrow[d] & (\Rscr^{\tame,I} \bullet \chi)_{f^0}(\pi^-_{\eta})_*(e^-_{\eta})^! \arrow[d] \arrow[l] \\ 
            (\pi^+_{\sigma})_!(e^+_{\sigma})^*(\Rscr^{\tame,I} \bullet \chi)_f \arrow[r]  & (\Rscr^{\tame,I} \bullet \chi)_{f^0}(\pi_{\eta}^+)_!(e^+_{\eta})^*.
        \end{tikzcd}
    \end{equation*}
Moreover, if $X$ is $S$-affine then all arrows are isomorphisms when we evaluate at equivariant objects. 
\end{prop}

\begin{proof}
In what follows, we color arrows in blue if it is an isomorphism as a consequence of the Braden hyperbolic localization theorem, Braden contraction lemma or Braden purity theorem. Note that we have 
\begin{equation*} 
\mathrm{HypLoc}(X^0_{\eta},I) = \mathrm{HypLoc}(X_{\eta},I) \times_{(X_{\eta},I)} (X_{\eta}^0,I)
\end{equation*} 
for trivial reason and $\mathrm{HypLoc}(\mathscr{R}_{X^0}^{\tame,I},I) = \mathrm{HypLoc}(X^0_{\eta},I) \times_{(X^0_{\eta},I)}(\mathscr{R}_{X^0}^{\tame,I},I)$ by proposition \ref{hyperbolic localization diagram of all specialization systems considered}, we deduce that 
\begin{equation*}
    \mathrm{HypLoc}(\mathscr{R}_{X^0}^{\tame,I},I) =  \mathrm{HypLoc}(X_{\eta},I) \times_{(X_{\eta},I)} (\mathscr{R}_X^{\tame,I},I).
\end{equation*}
Therefore, by theorem \ref{compatible of Braden transformations with four operations}, we have a commutative square
    \begin{equation*}
       \begin{tikzcd}[sep=large]
           (\pi^-)_*(e^-)^!(j_f \circ \theta_f)_* \arrow[d,color=blue] \arrow[r,"\simeq"] & (j_{f^0}\circ \theta_{f^0})_*(\pi^-)_*(e^-)^! \arrow[d,color=blue] \\ 
           (\pi^+)_!(e^+)^*(j_f \circ \theta_f)_*\arrow[r] & (j_{f^0}\circ \theta_{f^0})_*(\pi^+)_!(e^+)^*
       \end{tikzcd}
    \end{equation*}
Apply $(i_{f^0})^*$ and use theorem \ref{compatible of Braden transformations with four operations} again, we have commutative squares
  \begin{equation*}
       \begin{tikzcd}[sep=large]
          (\pi^-_{\sigma})_*(e^-_{\sigma})^!(i_f)^*(j_f \circ \theta_f)_* \arrow[d,color=blue] &  (i_{f^0})^*(\pi^-)_*(e^-)^!(j_f \circ \theta_f)_* \arrow[d,color=blue] \arrow[l] \arrow[r,"\simeq"] & (i_{f^0})^*(j_{f^0}\circ \theta_{f^0})_*(\pi^-)_*(e^-)^! \arrow[d,color=blue] \\ 
           (\pi^+_{\sigma})_!(e^+_{\sigma})^*(i_f)^*(j_f \circ \theta_f)_* & (i_{f^0})^*(\pi^+)_!(e^+)^*(j_f \circ \theta_f)_* \arrow[l,"\simeq",swap] \arrow[r] & (i_{f^0})^*(j_{f^0}\circ \theta_{f^0})_*(\pi^+)_!(e^+)^*
       \end{tikzcd}
    \end{equation*}
and hence a commutative square
\begin{equation*}
       \begin{tikzcd}[sep=large]
          (\pi^-_{\sigma})_*(e^-_{\sigma})^!i^*(j_f \circ \theta_f)_* \arrow[d] &  (i^0)^*(j_{f^0}\circ \theta_{f^0})_*(\pi^-)_*(e^-)^! \arrow[d] \arrow[l]  \\ 
           (\pi^+_{\sigma})_!(e^+_{\sigma})^*i^*(j_f \circ \theta_f)_* \arrow[r] & (i^0)^*(j_{f^0}\circ \theta_{f^0})_*(\pi^+)_!(e^+)^*.
       \end{tikzcd}
    \end{equation*}
Apply $(p_I)_{\#}$ to the diagram above and then we can use the third part of theorem \ref{compatible of Braden transformations with four operations} to obtain commutative squares
        \begin{equation*}
       \begin{tikzcd}[row sep=large]
      (\pi^-_{\sigma})_*(e^-_{\sigma})^!(p_I)_{\#}i^*(j_f \circ \theta_f)_* \arrow[d,color=blue] & (p_I)_{\#} (\pi^-_{\sigma})_*(e^-_{\sigma})^!i^*(j_f \circ \theta_f)_* \arrow[d,color=blue] \arrow[l] & (p_I)_{\#}(i^0)^*(j_{f^0}\circ \theta_{f^0})_*(\pi^-)_*(e^-)^! \arrow[d,color=blue] \arrow[l]  \\ 
       (\pi^+_{\sigma})_!(e^+_{\sigma})^*(p_I)_{\#}i^*(j_f \circ \theta_f)_*   &  (p_I)_{\#}(\pi^+_{\sigma})_!(e^+_{\sigma})^*i^*(j_f \circ \theta_f)_* \arrow[r,"\simeq"] \arrow[l,"\simeq",swap] & (p_I)_{\#}(i^0)^*(j_{f^0}\circ \theta_{f^0})_*(\pi^+)_!(e^+)^*
       \end{tikzcd}
    \end{equation*}
 and hence a commutative square
   \begin{equation*}
       \begin{tikzcd}[sep=large]
      (\pi^-_{\sigma})_*(e^-_{\sigma})^!(p_I)_{\#}i^*(j_f \circ \theta_f)_*\arrow[d,color=blue] & (p_I)_{\#}(i^0)^*(j_{f^0}\circ \theta_{f^0})_*(\pi^-)_*(e^-)^! \arrow[d,color=blue] \arrow[l]  \\ 
      (\pi^+_{\sigma})_!(e^+_{\sigma})^*(p_I)_{\#}i^*(j_f \circ \theta_f)_*  \arrow[r,"\simeq"] & (p_I)_{\#}(i^0)^*(j_{f^0}\circ \theta_{f^0})_*(\pi^+)_!(e^+)^*.
       \end{tikzcd}
    \end{equation*}
We apply $(\theta_f)^*(p_I)^*$ (which is smooth objectwise) and use the second part of theorem \ref{compatible of Braden transformations with four operations} and proposition \ref{proper base change theorem for diagrams} to get commutative squares
  \begin{equation*}
       \begin{tikzcd}[sep=large]
      (\pi^-_{\sigma})_*(e^-_{\sigma})^!(\Rscr^{\tame,I} \bullet \chi)_f \arrow[d,color=blue] &  (p_I)_{\#}(i^0)^*(j_{f^0}\circ \theta_{f^0})_*(\pi^-)_*(e^-)^!(\theta_f)^*(p_I)^* \arrow[d,color=blue] \arrow[l]  &  (\Rscr^{\tame,I} \bullet \chi)_{f^0}(\pi^-_{\eta})_*(e^-_{\eta})^! \arrow[l] \arrow[d,color=blue] \\ 
      (\pi^+_{\sigma})_!(e^+_{\sigma})^*(\Rscr^{\tame,I} \bullet \chi)_f \arrow[r,"\simeq"] & (p_I)_{\#}(i^0)^*(j_{f^0}\circ \theta_{f^0})_*(\pi^+)_!(e^+)^*(\theta_f)^*(p_I)^*  & (\Rscr^{\tame,I} \bullet \chi)_{f^0}(\pi^+_{\eta})_!(e^+_{\eta})^* \arrow[l,"\simeq"]
       \end{tikzcd}
    \end{equation*}
 and hence a desired commutative square 
       \begin{equation*}
        \begin{tikzcd}[sep=large]
            (\pi^-_{\sigma})_*(e^-_{\sigma})^!(\Rscr^{\tame,I} \bullet \chi)_f \arrow[d] & (\Rscr^{\tame,I} \bullet \chi)_{f^0}(\pi^-_{\eta})_*(e^-_{\eta})^! \arrow[d] \arrow[l] \\ 
            (\pi^+_{\sigma})_!(e^+_{\sigma})^*(\Rscr^{\tame,I} \bullet \chi)_f \arrow[r]  & (\Rscr^{\tame,I} \bullet \chi)_{f^0}(\pi_{\eta}^+)_!(e^+_{\eta})^*.
        \end{tikzcd}
    \end{equation*}
If moreover, $X$ is $S$-affine and we evaluate this diagram at equivariant objects, then all vertical arrows in those diagrams above become isomorphisms thanks to lemma \ref{equivariance of six operations}, and theorem \ref{braden theorem}, therefore it suffices to prove that for each square appears above, one of two horizontal arrows is an isomorphism, but this is what we have remarked at each diagram.
\end{proof}

To deal with general $X$ in case $(\Rscr^{\tame,I} \bullet \chi)$, we need the following proposition, which can be regarded as a descent result for specialization systems of this form.

\begin{prop} \label{descent of sp}
 Let $X,Y$ be algebraic $S$-spaces endowed with $\gm_{m,S}$-actions. Let $I \subset \Delta \times \mathbb{N}^{'\times}$ be a subcategory. Given equivariant morphisms $g \colon Y \longrightarrow X, f \colon X \longrightarrow \mathbb{A}_S^1$, then there is a commutative diagram
 \begin{equation*}
      \begin{tikzcd}[sep=large]
       (g^0_{\sigma})^* (\pi^+_{\sigma})_!(e^+_{\sigma})^*(\Rscr^{\tame,I} \bullet \chi)_f \arrow[r] \arrow[d] &    (g^0_{\sigma})^*(\Rscr^{\tame,I} \bullet \chi)_{f^0} (\pi_{\eta}^+)_!(e^+_{\eta})^* \arrow[d,"\alpha_{g^0}"] \\ 
       (\pi^+_{\sigma})_!(e^+_{\sigma})^*(g_{\sigma})^*(\Rscr^{\tame,I} \bullet \chi)_f  \arrow[d,"\alpha_g",swap] &  (\Rscr^{\tame,I} \bullet \chi)_{(f \circ g)^0}(g^0_{\eta})^*(\pi_{\eta}^+)_!(e^+_{\eta})^* \arrow[d] \\ 
       (\pi^+_{\sigma})_!(e^+_{\sigma})^*(\Rscr^{\tame,I} \bullet \chi)_{f \circ g}(g_{\eta})^* \arrow[r] & (\Rscr^{\tame,I} \bullet \chi)_{(f \circ g)^0}(\pi_{\eta}^+)_!(e^+_{\eta})^*(g_{\eta})^*.
     \end{tikzcd}
 \end{equation*}
 In in view of \cite[Corollaire 1.1.14]{ayoub-thesis-1}, the cube 
  \begin{equation*} 
       \begin{tikzcd}[sep=large]
        \bullet \arrow[rr,"L^+"] \arrow[rd,"(g_{\sigma})^*"] & &  \bullet \arrow[rd,"(g^0_{\sigma})^*"] & \\ 
           &  \bullet  & &  \bullet   \\ 
          \bullet \arrow[uu,"\sp_f"]  \arrow[rr,"L^+",dashed,pos=0.2] \arrow[rd,"(g_{\sigma})^*",swap]  & &  \bullet \arrow[rd,"(g^0_{\sigma})^*",dashed]  \arrow[uu,"\sp_{f^0}",pos=0.3,dashed]   & \\ 
            &  \bullet \arrow[uu,"\sp_{f \circ g}",pos=0.3,crossing over]  \arrow[rr,"L^+"] & &  \bullet \arrow[uu,"\sp_{(f\circ g)^0}",swap] \arrow[from=2-2,to=2-4,"L^+",pos=0.2,crossing over]
            \end{tikzcd}
   \end{equation*}
   is commutative.
\end{prop}
\begin{proof}
We prove only the commutativity of the first diagram, the second one is treated by a similar argument. Note that the base change morphism $\alpha_g$ can be given explicitly as 
 \begin{equation*}
    \begin{tikzcd}[sep=huge]
       \sh(X_{\eta}) \arrow[d,"(g_{\eta})^*",swap] \arrow[r,"(\theta_f)^*(p_I)^*"] & \shbb(\mathscr{R}_X,I) \arrow[r,"(j_f \circ \theta_f)_*"] \arrow[d,"(g_{\eta})^*"] \arrow[Rightarrow, shorten >=30pt, shorten <=30pt, dl,"="]  & \shbb(X,I) \arrow[d,"g^*"] \arrow[r,"(i_f)^*"] \arrow[Rightarrow, shorten >=30pt, shorten <=30pt, dl,"\Ex_*^*"]  &  \shbb(X_{\sigma},I) \arrow[d,"(g_{\sigma})^*"] \arrow[r,"(p_I)_{\#}"] \arrow[Rightarrow, shorten >=30pt, shorten <=30pt, dl,"="]  & \sh(X_{\sigma}) \arrow[d,"(g_{\sigma})^*"] \arrow[Rightarrow, shorten >=30pt, shorten <=30pt, dl,"\Ex_{\#}^*"]  \\ 
       \sh(Y_{\eta})   \arrow[r,"(\theta_{f \circ g})^*(p_I)^*"] & \shbb(\mathscr{R}_Y,I) \arrow[r,"(j_{f \circ g} \circ \theta_{f \circ g})_*"]  & \shbb(Y,I) \arrow[r,"(i_{f \circ g})^*"]   &  \shbb(Y_{\sigma},I) \arrow[r,"(p_I)_{\#}"]   & \sh(Y_{\sigma}).
    \end{tikzcd}
\end{equation*}
To simplify notation, we erase the subscript indices of operations $i_?,j_?,\theta_?$ and write $L^+$ for operations of the form $(\pi^+)_!(e^+)^*$. We can split the cube into four cubes 
 \begin{equation*}
     \begin{tikzcd}[sep=large]
        \bullet \arrow[rr,"L^+"] \arrow[rd,"(g_{\sigma})^*"] & &  \bullet \arrow[rd,"(g^0_{\sigma})^*"] & \\ 
           &  \bullet  & &  \bullet   \\ 
         \bullet \arrow[uu,"(p_I)_{\#}"]  \arrow[rr,"L^+",dashed,pos=0.2] \arrow[rd,"(g_{\sigma})^*",swap]  & &  \bullet \arrow[rd,"(g_{\sigma}^0)^*",pos=0.1]  \arrow[uu,"(p_I)_{\#}",pos=0.3,dashed]   & \\ 
            &  \bullet \arrow[uu,"(p_I)_{\#}",pos=0.3,crossing over,swap]  \arrow[rr,"L^+"] & &  \bullet \arrow[uu,"(p_I)_{\#}",swap] \arrow[from=2-2,to=2-4,"L^+",pos=0.2,crossing over]
     \end{tikzcd} \ \ \ \  \begin{tikzcd}[sep=large]
        \bullet \arrow[rr,"L^+"] \arrow[rd,"(g_{\sigma})^*"] & &  \bullet \arrow[rd,"(g^0_{\sigma})^*"] & \\ 
           &  \bullet & &  \bullet   \\ 
         \bullet \arrow[uu,"i^*"]  \arrow[rr,"L^+",dashed,pos=0.2] \arrow[rd,"(g_{\sigma})^*",swap]  & &  \bullet \arrow[rd,"(g_{\sigma}^0)^*",pos=0.1]  \arrow[uu,"i^*",pos=0.3,dashed]   & \\ 
            &  \bullet \arrow[uu,"i^*",pos=0.3,crossing over,swap]  \arrow[rr,"L^+"] & &  \bullet \arrow[uu,"i^*",swap] \arrow[from=2-2,to=2-4,"L^+",pos=0.2,crossing over] 
     \end{tikzcd} 
 \end{equation*}
 \begin{equation*}
    \begin{tikzcd}[sep=large]
        \bullet \arrow[rr,"L^+"] \arrow[rd,"g^*"] & &  \bullet \arrow[rd,"(g^0)^*"] & \\ 
           &  \bullet  & &  \bullet   \\ 
         \bullet \arrow[uu,"(j\circ\theta)_*"]  \arrow[rr,"L^+",dashed,pos=0.2] \arrow[rd,"(g_{\eta})^*",swap]  & &  \bullet \arrow[rd,"(g^0_{\eta})^*",pos=0.1]  \arrow[uu,"(j\circ\theta)_*",pos=0.3,dashed]   & \\ 
            &  \bullet \arrow[uu,"(j\circ\theta)_*",pos=0.3,crossing over,swap]  \arrow[rr,"L^+"] & &  \bullet \arrow[uu,"(j\circ\theta)_*",swap] \arrow[from=2-2,to=2-4,"L^+",pos=0.2,crossing over]
     \end{tikzcd} \ \ \ \   \begin{tikzcd}[sep=large]
        \bullet \arrow[rr,"L^+"] \arrow[rd,"(g_{\eta})^*"] & &  \bullet \arrow[rd,"(g_{\eta}^0)^*"] & \\ 
           &  \bullet  & &  \bullet   \\ 
         \bullet \arrow[uu,"(p_I)_{\#}"]  \arrow[rr,"L^+",dashed,pos=0.2] \arrow[rd,"(g_{\eta})^*",swap]  & &  \bullet \arrow[rd,"(g_{\eta}^0)^*",pos=0.1]  \arrow[uu,"(p_I)_{\#}",pos=0.3,dashed]   & \\ 
            &  \bullet \arrow[uu,"(p_I)_{\#}",pos=0.3,crossing over,swap]  \arrow[rr,"L^+"] & &  \bullet \arrow[uu,"(p_I)_{\#}",swap] \arrow[from=2-2,to=2-4,"L^+",pos=0.2,crossing over] 
     \end{tikzcd} 
 \end{equation*}
The commutativities of the second and the fourth cubes are clear because these cubes involve only one type of base change, namely, $\Ex_!^*$. Let us show the commutativity of the third cube. The commutativity of the first cube is proved similarly. By \cite[Corollaire 1.1.14]{ayoub-thesis-1}, we can show the commutativity of the cube
 \begin{equation*}
     \begin{tikzcd}[sep=large]
        \bullet \arrow[rr,"L^+"] \arrow[dd,"(j \circ \theta)^*",swap]  \arrow[rd,"g^*"] & &  \bullet \arrow[rd,"(g^0)^*"] & \\ 
           &  \bullet      & &  \bullet   \\ 
         \bullet  \arrow[rr,"L^+",dashed,pos=0.2] \arrow[rd,"(g_{\eta})^*",swap]  & &  \bullet \arrow[rd,"((g_{\eta})^0)^*",pos=0.1]  \arrow[uu,"(j\circ \theta)^*",pos=0.3,dashed]   & \\ 
            &  \bullet  \arrow[rr,"L^+"] & &  \bullet \arrow[uu,"(j\circ \theta)^*",swap] \arrow[from=2-2,to=2-4,"L^+",pos=0.2,crossing over]\arrow[from=2-2,to=4-2,"(j\circ \theta)^*",pos=0.7,crossing over]
     \end{tikzcd}
 \end{equation*}
 which is obvious. 
\end{proof}

\section{The motivic integral identity}
In this section, we deduce the categorical integral identity in $\sh$, which has a similar form of the original conjecture of Kontsevich and Soibelman. We stress to the fact that our result holds in both zero and positive characteristics. This property is clearly different from the virtual motives in Grothendieck rings of varieties as (to our best knowledge) one can only define nearby cycles in characteristic zero. We also follow \cite{florian-2024} to upgrade the identity to a monodromic version, and hence a study of quasi-unipotent motives of algebraic spaces is necessary. Finally, under realization, we recover the integral identity in \cite{bu-2024} in the ring $K_0(\qushct(X^0_{\sigma}))$. 
\subsection{The categorical motivic integral identity} Given a smooth morphism $f \colon V \longrightarrow X$ with a section $s \colon X \longrightarrow V$, we denote by $\th(f,s) = f_{\#}s_*$ the Thom operation together with $\th^{-1}(f,s) = s^!f^*$ its inverse. We have a lemma. 
\begin{lem} \label{sp systems commute with Thoms}
    Let $f \colon V \longrightarrow X$ be a smooth morphism of algebraic $S$-spaces and $s \colon X \longrightarrow V$ a section. Let $f \colon X \longrightarrow B$ be a morphism of algebraic $S$-spaces and $\sp$ be a specialization system over $(B,\eta,\sigma)$, then there exists a natural isomorphism 
    \begin{equation*}
        \th(f_{\sigma},s_{\sigma})\sp_f \overset{\sim}{\longrightarrow} \sp_f \th(f_{\eta},s_{\eta}).
    \end{equation*}
    Equivalently, there exists a natural isomorphism 
        \begin{equation*}
        \sp_f\th^{-1}(f_{\sigma},s_{\sigma}) \overset{\sim}{\longrightarrow} \th^{-1}(f_{\eta},s_{\eta})\sp_f.
    \end{equation*}
\end{lem}
\begin{proof}
Let $u \colon U \longrightarrow X$ be a $\nis$-covering. As in \cite[Proposition 3.1.7]{ayoub-thesis-2}, we reduce the question to the case when $X$ is a scheme, which is true by \cite[Proposition 3.1.7]{ayoub-thesis-2} and \cite[Theorem 5.9]{bang-2024}. 
\end{proof}

\begin{prop} \label{consequence of descent of sp}
Let $X$ be an algebraic $S$-space endowed with a $\gm_{m,S}$-action. Let $f \colon X \longrightarrow \mathbb{A}_S^1$ be an equivariant morphism. If the action on $X$ is $\tau$-locally linearizable, then the canonical morphisms  
\begin{align*}
     (\pi_{\sigma}^+)_!(e^+_{\sigma})^*(\Rscr^{\tame,I} \bullet \chi)_f(A) & \longrightarrow (\Rscr^{\tame,I} \bullet \chi)_{f^0} (\pi_{\eta}^+)_!(e^+_{\eta})^*(A) 
\end{align*}
are isomorphisms if $A$ is an equivariant object in $\sh(X_{\eta})$. In particular, if $\tau$ is the \'etale topology, then the motivic integral identity is always true. 
\end{prop}

\begin{proof}
Let $\left \{u_i \colon U_i \longrightarrow X \right \}$ be a $\gm_{m,S}$-equivariant $\tau$-covering of $X$ with $U_i$ being $S$-affine. By $\tau$-conservativity (see proposition \ref{descent of motives}), it suffices to prove that 
   \begin{equation*}
      ((u_i^0)_{\sigma})^*(\Rscr^{\tame,I} \bullet \chi)_{f^0} (\pi_{\eta}^+)_!(e^+_{\eta})^*(A) \longrightarrow ((u_i^0)_{\sigma})^* (\pi_{\sigma}^+)_!(e^+_{\sigma})^*(\Rscr^{\tame,I} \bullet \chi)_f(A)
   \end{equation*}
   is an isomorphism for any $u_i$. By proposition \ref{descent of sp}, there is a commutative diagram 
    \begin{equation*}
      \begin{tikzcd}[sep=large]
        ((u_i^0)_{\sigma})^* (\pi^+_{\sigma})_!(e^+_{\sigma})^*(\Rscr^{\tame,I} \bullet \chi)_f \arrow[r] \arrow[d] &    ((u_i^0)_{\sigma})^*(\Rscr^{\tame,I} \bullet \chi)_{f^0} (\pi_{\eta}^+)_!(e^+_{\eta})^* \arrow[d,"\alpha_{g^0}"] \\ 
       (\pi^+_{\sigma})_!(e^+_{\sigma})^*((u_i)_{\sigma})^*(\Rscr^{\tame,I} \bullet \chi)_f \arrow[d,"\alpha_g",swap] &  (\Rscr^{\tame,I} \bullet \chi)_{(f \circ u_i)^0}((u_i^0)_{\eta})^*(\pi_{\eta}^+)_!(e^+_{\eta})^* \arrow[d] \\ 
       (\pi^+_{\sigma})_!(e^+_{\sigma})^*(\Rscr^{\tame,I} \bullet \chi)_{f \circ u_i}((u_i)_{\eta})^* \arrow[r] & (\Rscr^{\tame,I} \bullet \chi)_{(f \circ u_i)^0}(\pi_{\eta}^+)_!(e^+_{\eta})^*((u_i)_{\eta})^*.
     \end{tikzcd}
 \end{equation*}
 In this diagrams, morphisms $\alpha_?$ are isomorphisms because $u_i$ are smooth, two other vertical arrows are isomorphisms thanks to theorem \ref{compatible of Braden transformations with four operations}, the lower horizontal arrow is an isomorphism by proposition \ref{compatible of Braden transformations with sp induced by derivators}.
\end{proof}

We obtain the categorical version of the motivic integral identity.

\begin{theorem} \label{categorical integral identity}
Let $X = \mathbb{A}_S^{d_1} \times_S \mathbb{A}_S^{d_2} \times_S Y$ be a $S$-variety with $Y$ a $S$-variety. Suppose that $\gm_{m,S}$ acts by positive weights on $\mathbb{A}^{d_1}_S$, by negative weights on $\mathbb{A}^{d_2}_S$ and trivially on $Y$, then the motivic integral identities
  \begin{equation*}
     (\pi_{\sigma}^+)_!(e^+_{\sigma})^*(\Rscr^{\tame,I} \bullet \chi)_f(\mathds{1}_{X_{\eta}}) \overset{\sim}{\longrightarrow}  \mathds{1}_{Y_{\sigma}}(-d_1)[-2d_1] \otimes (\Rscr^{\tame,I} \bullet \chi)_{f^0}(\mathds{1}_{X_{\eta}^0})
\end{equation*}
holds in $\sh(X^0_{\sigma})$. 
\end{theorem}

\begin{proof}
In this case, we can take any affine Nisnevich atlas of $Y$ (see \cite[Theorem 6.4]{knutson-1971}) and use proposition \ref{descent of motives} to see that the topology does not matter here. We denote $\sp$ to indicate $(\Rscr^{\tame,I} \bullet \chi)$. We see that
    \begin{align*}
          (\pi_{\sigma}^+)_!(e^+_{\sigma})^*\sp_f(\mathds{1}_{X_{\eta}}) & \simeq \sp_{f^0} (\pi_{\eta}^+)_!(e^+_{\eta})^*(\mathds{1}_{X_{\eta}}) &  \text{by proposition \ref{consequence of descent of sp}}  \\ 
          & \simeq \sp_{f^0}(s^+_{\eta})^!(\mathds{1}_{X^+_{\eta}}) & \text{by lemma \ref{braden contraction lemma}} \\ 
          & \simeq \sp_{f^0}(s^+_{\eta})^!(\pi^+_{\eta})^*(\mathds{1}_{X^0_{\eta}}) \\
          & \simeq (s^+_{\sigma})^!(\pi^+_{\sigma})^*\sp_{f^0}(\mathds{1}_{X^0_{\eta}})  & \text{by lemma \ref{sp systems commute with Thoms}} \\
          & \simeq \th^{-1}(\pi^+_{\sigma},s^+_{\sigma})(\mathds{1})  \otimes \sp_{f^0}(\mathds{1}_{X^0_{\eta}}) & \text{\cite[Section 2.4.13]{cisinski+deglise}} \\ 
         & \simeq \th^{-1}(\pi^+_{\sigma},s^+_{\sigma})(\mathds{1})  \otimes \sp_{f_{\mid Y}}(\mathds{1}_{Y_{\eta}}) 
    \end{align*}
and we conclude by observing that $X^0 = Y$ (embedded in $X$ via zero sections) and $\th^{-1}(\pi^+_{\sigma},s^+_{\sigma})(\mathds{1}) =  \mathds{1}_{Y_{\sigma}}(-d_1)[-2d_1]$ since in this case $\pi^+ \colon \mathbb{A}^{d_1}_Y \longrightarrow Y$ is the projection and $s^+ \colon Y \longhookrightarrow \mathbb{A}^{d_1}_Y$ is the zero section. 
\end{proof}
\subsection{Quasi-unipotent motives of algebraic spaces and monodromic categorical integral identity}
Now we would like to obtain the monodromic version of the theorem above. We extend the definition of quasi-unipotent motives (see \cite{ayoub-rigid-motive}\cite{florian+julien-2021}) to the case of algebraic spacess. Let $X$ be an algebraic $S$-space. The category $\qush(X)$ of \textit{quasi-unipotent motives} is a smallest triangulated subcategory of $\sh(\gm_{m,X})$ stable under all small direct sums and contains motives of the form $\mnor^{\vee}_n(Z,\bullet) = \pi_{\#}\mathds{1}_{Q^{\mathrm{gm}}_n(Z,\bullet)}(r)$, where
\begin{equation*}
    \pi \colon Q^{\mathrm{gm}}_n(Z,\bullet) \coloneqq \Spec\left(\frac{\mathcal{O}_Z[T,T^{-1},V]}{(V^n - \bullet T)} \right) \longrightarrow \Spec(\mathcal{O}_X[T,T^{-1}]) = \gm_{m,X}.
\end{equation*}
with $Z$ a scheme and $Z \longrightarrow X$ smooth, $\bullet \in \mathcal{O}_Z(Z)^{\times}$, $n \in \mathbb{N}^{\times}$. The category of constructible quasi-unipotent motives is defined as
\begin{equation*}
\qushct(X) = \shct(\gm_{m,X}) \cap \qush(X).
\end{equation*}
Let $f \colon X \longrightarrow Y$ be a morphism of algebraic $S$-spaces. If $X,Y$ are schemes, then by \cite[Proposition 4.8]{bang-2024} and \cite[Lemma 3.2.1]{florian+julien-2021}, the morphisms $(\breve{f})^*,(\breve{f})_!$ preserve quasi-unipotent motives and if $f$ is smooth, then $(\breve{f})_{\#}$ also preserves quasi-unipotent motives. If $X,Y$ are algebraic spaces, the same results hold for $(\breve{f})^*$ or $(\breve{f})_{\#}$. Let us prove that $(\breve{f})_!$ preserves quasi-unipotent motives, which is used in the subsequence. First we have the following corollary of proposition \ref{descent of motives}
\begin{lem} \label{compact generators of quasi-unipotent motives}
    Let $u \colon U \longrightarrow X$ be a $\tau$-atlas. Assume that $\sh(\gm_{m,U})$ is compactly generated by a set $\mathcal{S}$ of compact objects, then $\sh(\gm_{m,X})$ is compactly generated by the set $(\breve{u})_{\#}(\mathcal{S})$. Moreover, $\qush(X)$ is compactly generated by the set $\mathcal{S}(U,X)$ consisting of Tate twists of motives of the form $(\breve{u})_{\#}\big(\mnor^{\vee}_n(T,\bullet)\big)$ with $T \longrightarrow U$ a smooth morphism, $\bullet \in \mathcal{O}_T(T)^{\times}$, $n \in \mathbb{N}^{\times}$. 
\end{lem}

\begin{proof}
The first assertion is proposition \ref{compactly generated property of alg spaces}. The proof of proposition \ref{compactly generated property of alg spaces} also shows that motives of the form $(\breve{u})_{\#}\big(\mnor^{\vee}_n(T,\bullet)\big)$ are compact. It suffices to check that shifts of these motives detect zero object. Again, this is the same as the proof of proposition \ref{compactly generated property of alg spaces} because $(\breve{u})^*$ is conservative by propositon \ref{descent of motives} and $\qush(U)$ is compactly generated by those twists of $\mnor_n^{\vee}(T,\bullet)$. 
\end{proof}

\begin{prop} \label{proper morphism and quasi-unipotent motives}
    Let $f \colon X \longrightarrow Y$ be a proper morphism of algebraic $S$-spaces. If $A$ is in $\qush(X)$, then $(\breve{f})_*(A)$ is in $\qush(Y)$. Consequently, for any $f \colon X \longrightarrow Y$, $(\breve{f})_!$ preserves quasi-unipotent motives.
\end{prop}

\begin{proof}
The proof is very similar to the proof of \cite[Proposition 4.8]{bang-2024} by using the Chow lemma, cdh descent and Noetherian induction. The only case that we need to check is when $f$ is a closed immersion. Let us choose a cartesian square
\begin{equation*}
    \begin{tikzcd}[sep=large]
       U \arrow[d,"u",swap] \arrow[r,"g"] & V \arrow[d,"v"] \\ 
       X \arrow[r,"f"] & Y
    \end{tikzcd}
\end{equation*}
with $U,V$ being $\tau$-covering of $X,Y$, respectively. By \cite[Lemma 3.2.2]{florian+julien-2021}, the category $\qush(Y)$ is compactly generated by the set $(\breve{v})_{\#}(\mathcal{S}'(V,Y))$ with $\mathcal{S}'(V,Y)$ the set of Tate twists of motives of the form $(\breve{\iota})_*\big(\mnor^{\vee}_n(T,\bullet) \big)$ with $T \overset{h}{\longrightarrow} T'  \overset{\iota}{\longhookrightarrow} V$, $T,T'$ a scheme, $\iota$ closed immersion, $h$ smooth. We note that 
\begin{equation*}
    (\breve{f})_*(\breve{u})_{\#}(\mathcal{S}(U,X)) \subset (\breve{v})_{\#}(\mathcal{S}'(V,Y)). 
\end{equation*}
Indeed, this follows from the isomorphism $(\breve{f})_*(\breve{u})_{\#} \simeq (\breve{v})_{\#}(\breve{g})_*$. 
\end{proof}

Let $X$ be an algebraic $S$-space and $f \colon X \longrightarrow \mathbb{A}_S^1$ be a $S$-morphism. We form the diagram
\begin{equation*} \label{monodromic nearby cycles}
    \begin{tikzcd}[sep=large]
      \gm_{m,X_{\eta}} \arrow[d,"f_{\eta}^{\gm_m}",swap] \arrow[r]  & \gm_{m,X} \arrow[d,"f^{\gm_m}"] & \gm_{m,X_{\sigma}} \arrow[l] \arrow[d,"f_{\sigma}^{\gm_m}"] \\ 
     \gm_{m,S} \arrow[r,"j"] & \mathbb{A}_S^1 & S \arrow[l,"i",swap]
    \end{tikzcd}
\end{equation*}
where we define $f^{\gm_m}\colon \Spec(\mathcal{O}_X[T,T^{-1}]) \longrightarrow \mathbb{A}^1_S = \Spec(\mathcal{O}_S[t])$ to be the morphism given by $t \longmapsto fT$ once we identify $f$ with the image of $t$ under the canonical morphism $\mathcal{O}_S[t] \longrightarrow \mathcal{O}_X(X)$. With all these data, the \textit{monodromic nearby functor} is defined as
\begin{equation*}
    \Psi_f^{\mathrm{mon}} \coloneqq \Psi_{f^{\gm_m}}^{\tame}  p^*\colon \sh(X_{\eta}) \longmapsto \sh(\gm_{m,X_{\sigma}})
\end{equation*}
in which $p\colon \gm_{m,X_{\eta}} \longrightarrow X_{\eta}$ is the canonical projection. One can also define the \textit{monodromic unipotent nearby functor} $\Upsilon_f^{\textnormal{mon}}$ by a similar formula after replacing $\Psi^{\tame}_{f^{\gm_m}}$ with $\Upsilon_{f^{\gm_m}}$ in example \ref{ex: unipotent nearby functors}. Let $g \colon Y \longrightarrow X$ be a $S$-morphism, then there are base change morphisms 
\begin{equation*}
      \begin{tikzcd}[sep=large]
         \sh(X_{\eta}) \arrow[r,"\Psi_f^{\mathrm{mon}}"] \arrow[d,"(g_{\eta})^*",swap] & \sh(\gm_{m,X_{\sigma}}) \arrow[Rightarrow, shorten >=27pt, shorten <=27pt, dl,"\alpha_g"] \arrow[d,"(\breve{g}_{\sigma})^*"] \\ 
         \sh(Y_{\eta}) \arrow[r,"\Psi_{f \circ g}^{\mathrm{mon}}"] &  \sh(\gm_{m,Y_{\sigma}})
    \end{tikzcd} \ \ \ \  \begin{tikzcd}[sep=large]
         \sh(Y_{\eta}) \arrow[r,"\Psi_{f \circ g}^{\mathrm{mon}}"] \arrow[d,"(g_{\eta})_*",swap] & \sh(\gm_{m,Y_{\sigma}}) \arrow[d,"(\breve{g}_{\sigma})_*"] \\ 
         \sh(X_{\eta}) \arrow[Rightarrow, shorten >=27pt, shorten <=27pt, ur,"\beta_g"] \arrow[r,"\Psi_f^{\mathrm{mon}}"] &  \sh(\gm_{m,X_{\sigma}}).
         \end{tikzcd} 
\end{equation*}
which are isomorphisms if $g$ is smooth, proper, respectively. We can recover $\Psi_f$ from $\Psi^{\mathrm{mon}}_f$ thanks to \cite[Theorem 4.1.1]{florian+julien-2021} and \cite{bang-2024}. It was proved in \cite{florian+julien-2021} that $\Psi^{\mathrm{mon}}_f$ takes values in $\qush(X_{\sigma})$ if $X$ is a scheme (see \cite[Theorem 4.2.1]{florian+julien-2021}). A similar result holds for algebraic $k$-spaces with $k$ a field of characteristic zero.

\begin{prop}
 Let $S = \Spec(k)$ be the spectrum of a field of characteristic zero. Let $X$ be an algebraic $S$-space. Let $f \colon X \longrightarrow \mathbb{A}_S^1$ be a $S$-morphism and $A$ be an object in $\sh(X_{\eta})$, then $\Upsilon^{\mathrm{mon}}_f(A),\Psi^{\mathrm{mon}}_f(A)$ are in $\qush(X_{\sigma})$. 
\end{prop}

\begin{proof}
The result holds for schemes thanks to \cite[Theorem 4.2.1]{florian+julien-2021} (note that for schemes, the proof for $\Upsilon^{\mathrm{mon}}_f$ is similar to the case of $\Psi^{\mathrm{mon}}_f$ done in \cite{florian+julien-2021} so here we can just focus on $\Psi^{\mathrm{mon}}_f$). Let $u \colon U \longrightarrow X$ be a $\nis$-atlas of $X$. The category $\sh(U_{\eta})$ is compactly generated by Tate twists of motives of the form 
$(g_{\eta})_{\#}(\mathds{1}_{Y_{\eta}})$ with $g \colon Y \longrightarrow U$ a smooth morphism (see \cite[Théorème 4.5.67]{ayoub-thesis-2}). Hence, by corollary \ref{compactly generated property of alg spaces} and enlarging if necessary, the category $\sh(X_{\eta})$ is compactly generated by twists of motives of the form $(g_{\eta})_{\#}(\mathds{1}_{Y_{\eta}})$ with $g \colon Y \longrightarrow X$ a smooth morphism and $Y$ a scheme. Locally for $g$ smooth, $g_{\#}$ differs from $g_!$ some twists so $\sh(X_{\eta})$ is compactly generated by twists of motives of the form $(g_{\eta})_!(\mathds{1}_{Y_{\eta}})$ with $g \colon Y \longrightarrow X$ smooth and $Y$ a scheme. By the Nagata compactification theorem and localization property, we see that $\sh(X_{\eta})$ is compactly generated by twists of motives of the form $(g_{\eta})_*(\mathds{1}_{Y_{\eta}})$ with $g \colon Y \longrightarrow X$ proper and $Y$ a scheme (using the Chow lemma, one can even reduce to $g$ projective). Therefore, it is enough to show that
\begin{equation*}
    \Psi^{\mathrm{mon}}_f(g_{\eta})_*(\mathds{1}_{Y_{\eta}}) \simeq (\breve{g}_{\sigma})_*\Psi^{\mathrm{mon}}_{f \circ g}(\mathds{1}_{Y_{\eta}})
\end{equation*}
is in $\qush(X_{\sigma})$. This follows from the assumption on the case of schemes and proposition \ref{proper morphism and quasi-unipotent motives}.
\end{proof}

\begin{cor} \label{monodromic categorical integral identity}
Let $S = \Spec(k)$ be the spectrum of a field of characteristic zero. Let $Y$ be an algebraic $S$-space and $X = \mathbb{A}_S^{d_1} \times_S \mathbb{A}_S^{d_2} \times_S Y$. Suppose that $\gm_{m,S}$ acts by positive weights on $\mathbb{A}^{d_1}_S$, by negative weights on $\mathbb{A}^{d_2}_S$ and trivially on $Y$. Let $f \colon X \longrightarrow \mathbb{A}_S^1$ be a $\gm_{m,S}$-equivariant, where the target is equipped with the trivial action, then the monodromic motivic integral identities
    \begin{align*}
       (\breve{\pi}_{\sigma}^+)_!(\breve{e}^+_{\sigma})^*\Psi_f^{\mathrm{mon}}(\mathds{1}_{X_{\eta}}) & \simeq  \mathds{1}_{\gm_{m,Y_{\sigma}}}(-d_1)[-2d_1] \otimes \Psi_{f_{\mid Y}}^{\mathrm{mon}}(\mathds{1}_{Y_{\eta}}) \\ 
       (\breve{\pi}_{\sigma}^+)_!(\breve{e}^+_{\sigma})^*\Upsilon_f^{\mathrm{mon}}(\mathds{1}_{X_{\eta}}) & \simeq \mathds{1}_{\gm_{m,Y_{\sigma}}}(-d_1)[-2d_1] \otimes \Upsilon_{f_{\mid Y}}^{\mathrm{mon}}(\mathds{1}_{Y_{\eta}})
  \end{align*}
    hold in $\qush(X^0_{\sigma})$.
\end{cor}

\begin{proof}
This is an application of theorem \ref{categorical integral identity} to the morphism $f^{\gm_m}$ where we endow $\gm_{m,X} = \gm_{m,S} \times_S X$ with the diagonal action that is the trivial action on the first factor and the given action of $X$ on the second factor.    
\end{proof}
\subsection{The piecewise integral identity} In the rest of this subsection, we assume that $S = \Spec(k)$ is a field of characteristic zero. Let us recall the definition of $K_0(\var_X)$ with $X$ an algebraic $k$-space (see \cite{joyce-2007} or \cite{bu-2024}). The group $K_0(\var_X)$ is generated by isomorphisms classes of morphisms of algebraic $k$-spaces $U \longrightarrow X$ with $U$ a scheme of finite type over $k$ modulo the usual cut-paste relation. We denote by $\mathscr{M}_X$ the localization $K_0(\var_X)[\mathbf{L}^{-1}]$ with $\mathbf{L}=[\mathbb{A}_X^1]$ the class of the affine line. For any morphism of algebraic $k$-spaces $Y \longrightarrow X$, there exists a class $[Y]$ in $\mathscr{M}_X$, agreeing with $[Y]$ when $Y$ is a scheme. Given a triangulated category $\mathcal{T}$, we write $K_0(\mathcal{T})$ for its Grothendieck group with the set of generators being isomorphisms classes $[A]$ ($A \in \mathcal{T}$) of objects modulo the relation 
\begin{equation*}
    [B] = [A] + [C]
\end{equation*}
if there exists a distinguished triangle
\begin{equation*}
    A \longrightarrow B \longrightarrow C \longrightarrow +1.
\end{equation*}
The following lemma is an obvious analogue of \cite[Lemma 3.1]{florian+julien-2013}.
\begin{lem}
    Let $X$ be an algebraic $k$-space, there exists a unique ring homomorphism \begin{equation*}
        \chi_{X,c} \colon \mscr_X \longrightarrow K_0(\shct(X))
    \end{equation*}
    such that 
    \begin{equation*}
        \chi_{X,c}([U]) = [f_!(\mathds{1}_U)] 
    \end{equation*}
    for any morphism of algebraic $k$-spaces $f \colon U \longrightarrow X$ with $U$ a scheme.
\end{lem}
Analogous to the case of non-virtual motives, we can define the monodromic version of $\mathscr{M}_X$. Let $X$ be a $k$-variety. Follow \cite{bittner-2005}\cite{guibert+loeser+merle-2006}\cite{bu-2024}, we write $\mathscr{M}^{\hat{\mu}}_X$ the Grothendieck ring of $X$-varieties endowed with a $\hat{\mu}$-action with $\hat{\mu} = \varprojlim \mu_n$ the projective limit of roots of unity. One obtains an obvious analogue of \cite[Proposition 5.2.1]{florian+julien-2021}
\begin{lem}
    Let $X$ be an algebraic $k$-space, there exists a unique ring homomorphism 
    \begin{equation*}
        \chi_{X,c}^{\hat{\mu}} \colon \mscr_X^{\hat{\mu}} \longrightarrow K_0(\qushct(X))
    \end{equation*}
    such that 
    \begin{equation*}
        \chi_{X,c}([U]) = [f_!(\mathds{1}_U)] 
    \end{equation*}
    for any morphism of algebraic $k$-spaces $f \colon U \longrightarrow X$ with $U$ a scheme.
\end{lem}

By \cite{bu-2024} (for schemes, see for instance \cite{bittner-2005}), there is a nearby morphism 
\begin{equation*}
    \psi_f \colon \mscr_X \longrightarrow \mscr_{X_{\sigma}}^{\hat{\mu}}
\end{equation*}
satisfying smooth and proper base change theorems. 
\begin{lem}
    Let $X$ be an algebraic $k$-space and $f \colon X \longrightarrow \mathbb{A}_k^1$ be a morphism of algebraic $k$-spaces, then the diagram 
    \begin{equation*}
        \begin{tikzcd}[sep=large]
                    \mscr_X \arrow[r,"\psi_f"] \arrow[d,"\chi_{X,c}",swap]& \mscr_{X_{\sigma}}^{\hat{\mu}} \arrow[d,"\chi^{\hat{\mu}}_{X_{\sigma},c}"] \\ 
            K_0(\shct(X)) \arrow[r,"\Psi^{\mathrm{mon}}_f"] & K_0(\qushct(X_{\sigma}))
        \end{tikzcd}
    \end{equation*}
    is commutative. 
\end{lem}
\begin{proof}
By a piecewise version of proposition \ref{descent of motives}, we see that 
\begin{equation*} 
(u_{\sigma})^* \colon K_0(\qushct(U_{\sigma})) \longrightarrow K_0(\qushct(X_{\sigma}))
\end{equation*} 
is injective if $u \colon U \longrightarrow X$ is a Nisnevich atlas of $X$. The result then follow from \cite[Proposition 5.6]{bang-2024}.
\end{proof}

\begin{cor} \label{the virtual integral identity}
    Let $X = \mathbb{A}_k^{d_1} \times_k \mathbb{A}_k^{d_2} \times_k Y$ with $Y$ an algebraic $k$-space. Suppose that $\gm_{m,k}$ acts by positive weights on $\mathbb{A}^{d_1}_k$, by negative weights on $\mathbb{A}^{d_2}_k$ and trivially on $Y$, then the virtual integral identity
    \begin{equation*}
         \int_{\mathbb{A}^{d_1}_k}(\psi_f)_{\mid \mathbb{A}^{d_1}_k \times_k Y} = \mathbf{L}^{d_1}(\psi_{f\mid Y})
    \end{equation*}
    holds in $K_0(\qushct(X^0_{\sigma}))$.
\end{cor}

\begin{proof}
 One can copy the proof in \cite{florian-2024}. 
\end{proof}

\section{Motivic constant term functors}
The aim of this section is to extend the framework of motivic nearby cycles functors and Braden transformations to ind-schemes. As a consequence, we obtain a motivic enhancement of the compatibility of geometric constant term functors and Bernstein isomorphisms. 
\subsection{Extending specialization systems to ind-schemes}
In order to have the language to establish the next section, which involves ind-schemes like affine Grassmannians, here we study specialization systems on ind-schemes. Let $X = \colim_{i \in I} X_i$ be an ind-$S$-scheme. The transition morphisms induce adjunctions 
\begin{equation*}
    \big((\iota_{i \to j})_* = (\iota_{i \to j})_! \dashv (\iota_{i \to j})^! \big) \colon \sh(X_i) \longrightarrow \sh(X_j).  
\end{equation*}
We define a category $\sh(X)$ by the formula
\begin{equation*}
    \sh(X) \coloneqq \colim_{(\iota_{i \to j})_!} \sh(X_i)
\end{equation*}
where in the colimit, the transition morphisms are $(\iota_{i \to j})_!$. The category $\sh(X)$ is canonically a triangulated category and independent of the choice of the presentation $X = \colim_{i \in I}X_i$ by the same reason as in \cite[Lemma 8]{gaitsgory-2001}. 

\begin{lem} \label{well-definition of four operations of ind-schemes}
For any morphism $f$ of ind-$S$-schemes, the operations $(f_*,f_!,f^!)$ are well-defined. If moreover, $f$ is cartesian, then $f^*$ is well-defined. They satisfies the usual base change properties whenever they are defined.
\end{lem}
\begin{proof}
 Let us show that for instance, that we can define $f_!$. The other operations are treated similarly. An object $A \in \sh(X)$ coming from $\sh(X_i)$ is sent to $(f_i)_!(A)$ in $\sh(Y_i)$. We need to show that for $i \to j$ in $I$, there is an isomorphism
 \begin{equation*}
     (\iota_{i \to j}^Y)_*(f_i)_!(A) \simeq (f_j)_!(\iota_{i \to j}^X)_*(A),
 \end{equation*}
 but this is just the functoriality of the exceptional direct image. 
\end{proof}
\begin{prop} \label{extending sp systems to ind-schemes}
Let $f \colon X = \colim_{i \to I}X_i\longrightarrow B$ be a morphism of ind-$S$-schemes. Let $f_i \coloneqq X_i \longrightarrow B$ be induced morphisms of $S$-schemes. Let $\sp$ be a specialization system over $(B,i,j)$, then the functor $\sp_f \coloneqq \colim_{i \in I} \sp_{f_i} \colon \sh(X_{\eta}) \longrightarrow \sh(X_{\sigma})$ is well-defined. Moreover, for each morphism $g \colon Y \longrightarrow X$ of ind-$S$-schemes, the base change morphisms 
\begin{align*}
    \beta_g \colon \sp_f (g_{\eta})_* & \longrightarrow (g_{\sigma})_*\sp_{f \circ g} \\ 
    \nu_g \colon (g_{\sigma})_!\sp_{f \circ g} & \longrightarrow \sp_f (g_{\eta})_! \\ 
    \nu_g \colon \sp_{f \circ g}(g_{\eta})^!& \longrightarrow (g_{\sigma})^!\sp_f.
\end{align*}
are well-defined. If $g$ is cartesian, then the morphism 
\begin{align*}
    \alpha_g \colon (g_{\sigma})^*\sp_f & \longrightarrow \sp_{f \circ g}(g_{\eta})^*
\end{align*}
is well-defined.
\end{prop}
\begin{proof}
 We define $\sp_f(A)$ to be $\sp_{f_i}(A)$ if $A$ comes from $\sh((X_i)_{\eta})$. To prove that $\sp_f$ is well-defined, we have to show that if $i \to j$, then $(\iota_{i \to j,\sigma})_*(\sp_{f_i}(A)) \simeq \sp_{f_j}(\iota_{i \to j,\eta})_*(A)$, but this follows from the definition. Let us explain how to define base change morphisms $\beta_g$, the others are done using the analogous arguments. We assume that $g = \colim_{i \in I} g_i \colon Y = \colim_{i \in I}Y_i \longrightarrow X = \colim_{i \in I}X_i$ is colimit of $g_i \colon Y_i \longrightarrow X_i$.  If $A \in \sh((Y_{i,\eta})$, we define $\beta_g(A)$ to be $\sp_{f_i}(g_{i,\eta})_*(A) \longrightarrow (g_{i,\sigma})_*\sp_{f_i \circ g_i}(A)$. This is well-defined as for $i \to j$ in $I$, there is a commutative diagram 
 \begin{equation*}
     \begin{tikzcd}[sep=large]
         (\iota_{i \to j,\sigma})_*\sp_{f_i}(g_{i,\eta})_*(A) \arrow[r] \arrow[d,"\simeq",swap] & (\iota_{i \to j,\sigma})_*(g_{i,\sigma})_*\sp_{f_i \circ g_i}(A) \arrow[r,equal] & (f_{j,\sigma})_*(\iota_{i \to j,\sigma})_*\sp_{f_i \circ g_i}(A) \arrow[dd,"\simeq"]  \\ 
         \sp_{f_j}(\iota_{i \to j,\eta})_*(g_{i,\eta})_*(A) \arrow[d,equal] & & \\
         \sp_{f_j}(f_{j,\eta})_*(\iota_{i \to j,\eta})_*(A) \arrow[rr] & & (f_{j,\sigma})_*\sp_{f_i}(\iota_{i \to j,\eta})_*(A)
     \end{tikzcd}
 \end{equation*}
 with indicated isomorphisms. Let $A \longrightarrow B$ be a morphism in $\sh(Y_{\eta})$. Suppose that $A \in \sh(Y_{i,\eta})$ and $B \in \sh(Y_{j,\eta})$ and without loss of generality, we can assume that $i \longrightarrow j$, hence a morphism $A \longrightarrow B$ is simply a morphism $(\iota_{i \to j,\eta})_*(A) \longrightarrow B$ in $\sh(Y_{j,\eta})$. The rest now follows from the naturality of $\sp$. 
\end{proof}

\subsection{Extending Braden transformations to ind-schemes}
 In geometric Langlands, it is natural to work with ind-schemes and we want to have analogous results of the preceding section for such objects. The reader feels free to extend results in this section to ind-algebraic spaces. Let $R$ be a ring and $S = \Spec(R)$ its spectrum. Let $X = \colim_{i \in I}X_i$ be an ind-$S$-variety endowed with a $\tau$-locally linearizable $\gm_{m,S}$-action (where $\tau$ is at least the \'etale topology), this means there exists a presentation $X = \colim_{i \in I} X_i$ where the $\gm_{m,S}$-action on each $X_i$ is $\tau$-locally linearizable in the sense of of schemes, see for instance \cite[Definition 1.6]{richarz-2018}. Then \cite[Theorem 2.1]{richarz-2021} proves that the morphism $X^+ \longrightarrow X$ is cartesian and hence by lemma \ref{well-definition of four operations of ind-schemes}, the operation $(e^+)^*$ is well-defined. 
 \begin{prop} \label{braden theorem for ind-schemes}
     Let $X$ be an ind-$S$-variety endowed with a $\tau$-locally linearizable $\gm_{m,S}$-action, then the Braden transformation
     \begin{equation*}
         (\pi^-)_*(e^-)^! \longrightarrow (\pi^+)_!(e^+)^*
     \end{equation*}
     is well-defined and is an isomorphism on equivariant objects. 
 \end{prop}
 
\begin{proof}
It suffices to prove a generalization of lemma \ref{compatible of Braden transformations with closed immersions}: let $Z,X$ be $S$-varieties endowed with $\tau$-locally linearizable $\gm_{m,S}$-actions and $z \colon Z \longhookrightarrow X$ be an equivariant closed immersion. Then there is a commutative diagram 
    \begin{equation*}
       \begin{tikzcd}[sep=large]
           (\pi^-)_*(e^-)^!z_* \arrow[d,"\phi_X",swap] \arrow[r,"(\Phi^-)_*"] & (z^0)_*(\pi^-)_*(e^-)^! \arrow[d,"\phi_Z"] \\ 
           (\pi^+)_!(e^+)^*z_* \arrow[r,"(\Phi^+)_*"] & (z^0)_*(\pi^+)_!(e^+)^*
       \end{tikzcd}
    \end{equation*}
    and the two vertical arrows are isomorphisms. Indeed, by the proof of \cite[Theorem 2.1]{richarz-2021}, we know that $Z^0 \longrightarrow X^0$ and $Z^{\pm} \longrightarrow X^{\pm}$ are closed immersions. To repeat the proof of \cite[Lemma 2.23]{richarz-2018}, it suffices to show that $Z^0 = Z^{\pm} \times_{X^{\pm}} X^0$. We choose an affine, equivariant $\tau$-covering $U \longrightarrow X$, then $U \times_X Z \longrightarrow X$ is an affine, equivariant $\tau$-covering as well. Since being an isomorphism is fpqc local on target, we see that $Z^0 = Z^{\pm} \times_{X^{\pm}} X^0$ is true since it is true after base changing to $U$. 
\end{proof}
 \begin{ex} \label{constant term functors}
 Let $G$ be an algebraic group over a field $k$, assumed to be algebraically closed for simplicity. The \textit{affine Grassmannian} $\operatorname{Gr}_G$ is defined to be the fpqc sheaf associated with the functor 
 \begin{equation*}
     A \longmapsto G\big(A(\!(t)\!) \big)/G\big(A[\![t]\!]\big)
 \end{equation*}
with $A$ varying over $k$-algebras. It is well-known to be an ind-scheme and ind-projective if $G$ is reductive (see for instance \cite{mirkovic+vilonen-2018}\cite{beilinson+drinfeld-1996} and see also \cite{richarz-2019}\cite{zhu-2016}\cite{achar-secondbook}). Let us assume that $G$ is connected reductive over $k$. Let $T \subset B\subset G$ be a Borel subgroup $B$ and a maximal torus $T$. Let $B \subset P \subset G$ be a parabolic subgroup containing $B$ and let $T \subset M \subset P$ be a Levi subgroup containing $T$. These data induce
\begin{equation*}
    \mathrm{Gr}_M \overset{\pi^+}{\longleftarrow} \mathrm{Gr}_P \overset{e^+}{\longrightarrow} \mathrm{Gr}_G.
\end{equation*}
The connected components $\pi_0(\mathrm{Gr}_M)$ are canonically identified with $X_*(T)/\mathbb{Z}R^{\vee}(M,T)$. Let $\rho_G$ be half-sum of positive roots of $G$ determined by $B$. Simiarly, let $\rho_M$ be half-sum of positive roots of $M$ determined by $B \cap M$. As $\left <2\rho_G - 2\rho_M,\lambda \right>=0$ for any $\lambda \in R^{\vee}(M,T)$, we can define a locally constant function
\begin{equation*}
    \deg_M \colon \mathrm{Gr}_M \longrightarrow \pi_0(\mathrm{Gr}_M) \overset{\left <2\rho_G - 2\rho_M,- \right >}{\longrightarrow} \mathbb{Z}.
\end{equation*}
We then define the \textit{constant term functor} as
\begin{equation*}
    \operatorname{CT}_M^G \coloneqq (\pi^+)_!(e^+)^*[\deg_M] \colon \sh(\mathrm{Gr}_G) \longrightarrow \sh(\mathrm{Gr}_M)
\end{equation*}
where one puts $\deg_M$ here for perversity reason (see the example below). Constant term functors are well-known in geometric Langlands, see for instance \cite{beilinson+drinfeld-1996}\cite{braverman+gaitsgory-2002}\cite{mirkovic+vilonen-2018}\cite{cass+pepin-2024}. Particular cases of constant term functors are those induced by cocharacters. Let $\lambda  \colon \gm_{m,k} \longrightarrow G$ be a cocharacter. The cocharacter $\lambda$ induces via conjugation a $\gm_{m,k}$-action on $G$ as follows
\begin{equation*}
    \gm_{m,k} \times G \longrightarrow G, \ \ (t,g) \longmapsto \lambda(t)g\lambda(t)^{-1}.
\end{equation*}
With this action, $P^{\pm} \coloneqq G^{\pm}$ are opposite parabolic subgroups defined by dynamic method in \cite{conrad-2014} and $M = P^+ \cap P^- = G^0$ is a common Levi. This gives rise to a Braden transformation 
\begin{equation*}
  (\pi^-)_*(e^-)^![\deg_M] \longrightarrow  (\pi^+)_!(e^+)^*[\deg_M]
\end{equation*}
between functors $\sh(\mathrm{Gr}_G) \longrightarrow \sh(\mathrm{Gr}_M)$. The action is even Zar-locally linearizable by \cite[Lemma 3.3]{richarz-2021} so the Braden transformation above is an isomorphism on equivariant objects. 
\end{ex}

\begin{ex}
In \cite{florian+morel-2019} (and subsequent works such as \cite{terenzi-2025}\cite{tubach-2025}), Ivorra and Morel develops the theory of \textit{perverse Nori motives} (also called perverse motivic sheaves) $\mathscr{M}\operatorname{Perv}(-)$ and show that they acquire a four-functor formalism $(f^*,f_*,f_!,f^!)$ (the two other functors $(\otimes,\underline{\Hom})$ are constructed in \cite{terenzi-2025}) on the bounded derived categories $\mathbf{DN}^b(-) = \mathbf{D}^b(\mathscr{M}\operatorname{Perv}(-))$. The Braden transformation can be defined analogously as above between constant term functors on Nori motives, namely, there are constant term functors $\mathbf{DN}^b(\mathrm{Gr}_G) \longrightarrow \mathbf{DN}^b(\mathrm{Gr}_M)$. It is natural to expect certain perversity of this functor. For instance, in \cite{pham-2026}, the author builds the equivariant Nori motives $\mathscr{M}\operatorname{Perv}_{\mathrm{L}^+G}(\mathrm{Gr}_G), \mathbf{DN}^b_{\mathrm{L}^+G}(\mathrm{Gr}_G)$ (where $\mathrm{L}^+G(A) = G\big(A[\![t]\!]\big)$ is the positive loop group) so that we expect that the constant term functors produce a functor $\mathscr{M}\operatorname{Perv}_{\mathrm{L}^+G}(\mathrm{Gr}_G) \longrightarrow \mathbf{DN}^b_{\mathrm{L}^+M}(\mathrm{Gr}_M)$ (where both sides should be understood as stratified perverse Nori motives as in \cite{pham-2026}) that is monoidal with respect to the convolution product on both sides (see \cite{pham-2026} for the definition of convolution products on perverse Nori motives of affine Grassmannians) and moreover, the image is an equivariant perverse Nori motives as in the classical setting (see \cite[Lemma 15.1]{baumann+riche-2018} for instance).
\end{ex}

\subsection{Compatibility of geometric constant term functors and Bernstein isomorphisms}
Let us assume from now on that $R$ is the ring of integers of a local field $F$.
\begin{prop}
Let $X$ be an ind-$S$-variety endowed with a $\tau$-locally linearizable $\gm_{m,S}$-action. Let $f \colon X \longrightarrow \mathbb{A}^1_S$ be an equivariant morphism, then there exists a commutative diagram 
   \begin{equation*}
        \begin{tikzcd}[sep=large]
            (\pi^-_{\sigma})_*(e^-_{\sigma})^!\Psi_f(A) \arrow[d] & \Psi_{f^0}(\pi^-_{\eta})_*(e^-_{\eta})^!(A) \arrow[d] \arrow[l] \\ 
            (\pi^+_{\sigma})_!(e^+_{\sigma})^*\Psi_f(A) \arrow[r]  & \Psi_{f^0} (\pi_{\eta}^+)_!(e^+_{\eta})^*(A)
        \end{tikzcd}
    \end{equation*}
    natural in $A \in \sh(X_{\eta})$. Moreover, then all arrows are isomorphisms providedt that $A$ is equivariant.
\end{prop}

\begin{proof}
The proof is similar to the proof of \cite[Theorem 6.1]{richarz-2021}. We remind the proof for the convenience of the reader. At first we note that $\Psi_f$ and Braden transformations are well-defined thanks to propositions \ref{extending sp systems to ind-schemes}, and \ref{braden theorem for ind-schemes}. The morphisms $(e^-_?)^!$ also make sense because they are cartesian thanks to \cite[Theorem 2.1]{richarz-2021}. All arrows are constructed by taking limits of those coming from schemes. Let $A \in \sh(X_{\eta})$, then it suffices to assume that $A \in \sh(X_{i,\eta})$ for some $i$. The result follows from proposition \ref{compatible of Braden transformations with sp induced by derivators}.
\end{proof}

By taking homotopy colimits, we obtain an analogue of \cite[Theorem 6.1]{richarz-2021}.
\begin{cor} \label{richarz-thm61-analogue}
  Regarding the hypothesis in \ref{total nearby functors}, let $X$ be a separated ind-$R$-scheme of ind-finite type endowed with a $\et$-locally linearizable $\gm_{m,R}$-action. Let $A$ be an equivariant object in $\da(X_{\eta},\mathbb{Q})$. 
  There exists a commutative diagram
    \begin{equation*}
        \begin{tikzcd}[sep=large]
            (\pi^-_{\overline{\sigma}})_*(e^-_{\overline{\sigma}})^!\Psi_f^{\mathrm{tot}}\arrow[d] & \Psi_{f^0}^{\mathrm{tot}}(\pi^-_{\eta})_*(e^-_{\eta})^! \arrow[d] \arrow[l] \\ 
            (\pi^+_{\overline{\sigma}})_!(e^+_{\overline{\sigma}})^*\Psi_f^{\mathrm{tot}} \arrow[r]  & \Psi_{f^0}^{\mathrm{tot}} (\pi_{\eta}^+)_!(e^+_{\eta})^*
        \end{tikzcd}
    \end{equation*}
    all arrows are isomorphisms when computed at equivariant objects. 
\end{cor}

\begin{rmk}
    The corollary above shows that \cite[Theorem 6.1]{richarz-2021} (see also \cite[Theorem 3.3]{richarz-2018}) is motivic in the following sense: in \cite{ayoub-realisation-etale}, Ayoub constructs an \'etale realization functor 
    \begin{equation*}
        \mathfrak{R}^{\et} \colon \da_{\mathrm{ct}}(-,\mathbb{Q}) \longrightarrow  \mathbf{D}^b_{\mathrm{ct}}(-,\overline{\mathbb{Q}}_{\ell}) 
    \end{equation*}
    commuting with four operations $(f^*,f_*,f_!,f^!)$ of schemes. Moreover, if \cite[Hypothèse 7.3]{ayoub-realisation-etale} is satisfied (which is automatic in the situation of example \ref{total nearby functors}), then $\mathfrak{R}^{\et}$ commutes with nearby cycles functors (see \cite[Théorème 10.16]{ayoub-realisation-etale}) in \'etale cohomology. For a separated $R$-ind-scheme of ind--finite type endowed with a $\gm_{m,R}$-action $X = \colim_{i \in I}X_i$, we define 
    \begin{equation*}
        \mathbf{D}^b_{\mathrm{ct}}(X,\overline{\mathbb{Q}}_{\ell})  = \colim_{i \in I}\mathbf{D}^b_{\mathrm{ct}}(X_i,\overline{\mathbb{Q}}_{\ell}) 
    \end{equation*}
    then under $\mathfrak{R}^{\et}$, the diagram in \ref{richarz-thm61-analogue} becomes the diagram in \cite[Theorem 6.1]{richarz-2021}.
\end{rmk}
We would like to establish a motivic version of \cite[Theorem A]{richarz-2021}, which is an analogue of the compatibility of Bernstein isomorphism with constant term map (see \cite{haines-2014}). We fix a triplet $(G,\mathcal{G},\lambda)$ where $G$ is a connected reductive $F$-group, $\mathcal{G}$ is a parahoric $R$-group with generic fiber $G$ and $\lambda \colon \gm_{m,R} \longrightarrow \mathcal{G}$ be a cocharacter over $R$. The cocharacter $\lambda$ acts on $\mathcal{G}$ via conjugation. Let $\mathcal{M} = \mathcal{G}^0$ and $\mathcal{P}^{\pm} = \mathcal{G}^{\pm}$ be the fixed points and repeller, attractor. The generic fiber $M = \mathcal{M}_F$ is a $F$-Levi subgroup of $G$ and $P^{\pm} = \mathcal{P}^{\pm}_F$ are parabolic subgroups with $P^+ \cap P^- = M$. The natural morphisms 
\begin{equation*}
    \mathcal{M} \overset{\pi^{\pm}}{\longleftarrow} \mathcal{P}^{\pm} \overset{e^{\pm}}{\longrightarrow} \mathcal{G}
\end{equation*}
give rise to morphisms of Beilinson-Drinfeld Grassmannians (see for instance \cite{zhu-2014}\cite{richarz-2016}\cite{richarz-2021} and see also \cite{achar-secondbook})
\begin{equation*}
   \mathrm{Gr}_{\mathcal{M}}^{\textnormal{BD}} \overset{\pi^{\pm}}{\longleftarrow} \mathrm{Gr}_{\mathcal{P}^{\pm}}^{\textnormal{BD}} \overset{e^{\pm}}{\longrightarrow} \mathrm{Gr}_{\mathcal{G}}^{\textnormal{BD}}
\end{equation*}
whose generic fiber and special fiber are 
\begin{align*}
     \mathrm{Gr}_M \overset{\pi^{\pm}_{\eta}}{\longleftarrow} \mathrm{Gr}_{P^{\pm}} \overset{e^{\pm}_{\eta}}{\longrightarrow} \mathrm{Gr}_G\\ 
    \flag_{\mathcal{M}} \overset{\pi^{\pm}_{\sigma}}{\longleftarrow} \flag_{\mathcal{P}^{\pm}} \overset{e^{\pm}_{\sigma}}{\longrightarrow} \flag_{\mathcal{G}}
\end{align*}
affine Grassmannians and affine flag varieties, respectively. Hence, there are two total nearby functors (see example \ref{total nearby functors})
\begin{align*}
    \Psi_{\mathcal{M}}^{\mathrm{tot}} \colon \da(\mathrm{Gr}_M,\mathbb{Q}) & \longrightarrow \da(\flag_{\mathcal{M},\overline{\sigma}},\mathbb{Q}) \\ 
    \Psi_{\mathcal{G}}^{\mathrm{tot}} \colon \da(\mathrm{Gr}_G,\mathbb{Q}) & \longrightarrow \da(\flag_{\mathcal{G},\overline{\sigma}},\mathbb{Q})  
\end{align*}
and geometric constant term functors (see example \ref{constant term functors}) 
\begin{align*}
    \mathrm{CT}_M \coloneqq (\pi_{\eta}^+)_!(e^+_{\eta})^* \colon \da(\mathrm{Gr}_G,\mathbb{Q}) & \longrightarrow \da(\mathrm{Gr}_M,\mathbb{Q}) \\ 
    \mathrm{CT}_{\mathcal{M}} \coloneqq (\pi_{\overline{\sigma}}^+)_!(e_{\overline{\sigma}}^+)^* \colon \da(\flag_{\mathcal{G},\overline{\sigma}},\mathbb{Q}) & \longrightarrow \da(\flag_{\mathcal{M},\overline{\sigma}},\mathbb{Q}) 
\end{align*}
Here is main theorem of this section that we have for free thanks to what we have built in previous sections
\begin{theorem}
There is a natural transformation 
\begin{equation*}
    \mathrm{CT}_{\mathcal{M}} \circ  \Psi^{\mathrm{tot}}_{\mathcal{G}} \longrightarrow \Psi^{\mathrm{tot}}_{\mathcal{M}} \circ \mathrm{CT}_M 
\end{equation*}
which is an isomorphism on $\gm_{m,K}$-equivariant objects. 
\end{theorem}

We end this section by an example
\begin{ex} 
Given the previous section, we can carry some computations motivated by the local model of the group $G = \operatorname{GL}_{2,\mathbb{Q}_p}$ in \cite{richarz-2021}. Let $S = \Spec(\mathbb{Z}_p)$ and let $f \colon X \longrightarrow S$ be the projective $S$-scheme obtained by blowing up $\mathbb{P}_{\mathbb{Z}_p}^1$ at the point $\left \{0 \right \}_{\mathbb{F}_p}$ of the special fiber. Hence, the generic fiber of $X$ is $\mathbb{P}^1_{\mathbb{Q}_p}$ while the special fiber is the intersection of two copies of $\mathbb{P}^1_{\mathbb{F}_p}$ meeting transversally at a point denoted by $e_{\sigma}$. The $\gm_{m,\mathbb{Z}_p}$-action $\lambda \cdot [x : y] = [\lambda x : y]$ on $\mathbb{P}_{\mathbb{Z}_p}^1$ lifts to a unique action on $X$. The $\mathbb{Q}_p$-points $0_{\eta}$ and $\infty_{\eta}$ lift uniquely to $\mathbb{Z}_p$-points $0_S$ and $\infty_S$ by the valuative criterion for properness applied to the diagram
    \begin{equation*}
        \begin{tikzcd}[sep=large]
           \eta \arrow[r,"0_{\eta}\text{,}\infty_{\eta}"] \arrow[d] & X_{\eta} \arrow[r] & X \arrow[d] \\
           S \arrow[rr,equal] \arrow[rru,dashed] & & S.
        \end{tikzcd}
    \end{equation*}
    We see that
    \begin{equation*}
         X^0  = 0_S \sqcup \infty_S \sqcup e_{\sigma} \ \ \ \ \text{and} \ \ \ \   X^{\pm}  = (\mathbb{A}^1_S)^{\pm} \sqcup \infty_S \sqcup (\mathbb{A}^1_{\sigma})^{\pm}.
    \end{equation*}
    The morphisms $\pi^+,\pi^-$ are given by contracting to fixed point. One can write $X^+ = \mathbb{A}_{0_S \sqcup e_{\sigma}}^1 \sqcup \infty_S$ with two inclusions $i_1 \colon \mathbb{A}^1_{0_S \sqcup e_{\sigma}} \longhookrightarrow X^+$ and $i_2 \colon \infty_S \longhookrightarrow X^+$. We also denote $i_3 \colon U = 0_S \sqcup e_{\sigma} \longhookrightarrow X^0$ and $i_4 \colon V= \infty_S \longhookrightarrow X^0$ canonical inclusions. We denote by $p \colon \mathbb{A}^1_{0_S \sqcup e_{\sigma}} \longrightarrow 0_S \sqcup e_{\sigma}$ the canonical projection. The motive $A = \mathds{1}_{X_{\eta}}$ is equivariant and we can compute
    \begin{align*}
        L^{\pm}_{X_{\eta}}(A) & \simeq  (\pi_{\eta}^+)_!(e^+_{\eta})^*(\mathds{1}_{X_{\eta}}) \simeq (\pi_{\eta}^+)_!(\mathds{1}_{X^+_{\eta}}) \\ 
        & \simeq (\pi_{\eta}^+)_!(i_{1,\eta})_!(\mathds{1}_{\mathbb{A}^1_{0_\eta}})  \oplus (\pi_{\eta}^+)_!(i_{2,\eta})_!(\mathds{1}_{\infty_{\eta}}) \\ 
        & \simeq (i_{3,\eta})_!(p_{\eta})_!(\mathds{1}_{\mathbb{A}^1_{0_\eta}}) \oplus (i_{4,\eta})_!(\mathds{1}_{\infty_{\eta}}) \\
        & \simeq (i_{3,\eta})_!(\mathds{1}_{0_{\eta}})(-1)[-2] \oplus (i_{4,\eta})_!(\mathds{1}_{\infty_{\eta}}).
    \end{align*}
   Since $\Psi_S^{\mathrm{tot}} \colon \da(\eta,\mathbb{Q}) \longrightarrow \da(\overline{\sigma},\mathbb{Q})$ is unitary monoidal by \cite[Théorème 10.19]{ayoub-realisation-etale}, we see that
    \begin{align*}
       (\pi_{\overline{\sigma}}^+)_!(e^+_{\overline{\sigma}})^*\Psi_X^{\mathrm{tot}}(A) \simeq \Psi_{X^0}^{\mathrm{tot}}(L^+_{X_{\eta}}(A)) & \simeq \Psi_{X^0}^{\mathrm{tot}}(i_{3,\eta})_!(\mathds{1}_{0_{\eta}})(-1)[-2] \oplus \Psi_{X^0}^{\mathrm{tot}}(i_{4,\eta})_!(\mathds{1}_{\infty_{\eta}}) \\
        & \simeq  (i_{3,\overline{\sigma}})_!\Psi_U^{\mathrm{tot}}(\mathds{1}_{0_{\eta}})(-1)[-2] \oplus (i_{4,\overline{\sigma}})_!\Psi_V^{\mathrm{tot}}(\mathds{1}_{\infty_{\eta}}) \\ 
        & \simeq  (i_{3,\overline{\sigma}})_!(\mathds{1}_{0_{\overline{\sigma}}})(-1)[-2] \oplus (i_{4,\overline{\sigma}})_!(\mathds{1}_{\infty_{\overline{\sigma}}}).
    \end{align*}
    Under the $\ell$-adic realization and at the level of $\ell$-adic cohomology, the identity above reads
    \begin{equation*}
        \mathbb{H}^*_c(X_{\overline{\sigma}}^+,\Psi^{\mathrm{tot}}_X(A)) \simeq  \overline{\mathbb{Q}}_{\ell}(-1)[-2] \oplus \overline{\mathbb{Q}}_{\ell}
    \end{equation*}
    which is compatible with \cite[Example 3.5]{richarz-2018}\cite[Section 1.3.2]{richarz-2021}. 
\end{ex} 
    \begin{rmk}
        One can perform the same computation with $S = \Spec(k[\![t ]\!])$, $k$ a field of characteristic zero and $\Psi^{\textnormal{tot}}$ is rep $t$-adic nearby functor considered in example \ref{adic nearby functors}. The importance is that $\Psi^{\adic}_{\id}$ is unitary monoidal. 
    \end{rmk}
\section{Appendix: Four-functor formalism of algebraic spaces}
In \cite{chowdhury-2024}\cite{chowdhury+angelo-2024}\cite{khan+ravi-2024}, the four-functor formalism for nice algebraic stacks is constructed for the case of $\mathbf{SH}$, the motivic stable homotopy category of schemes (which is $\sh$ below with $\tau$ the Nisnevich topology and $\mfrak = \spect^{\Sigma}_{S^1}(\sets^{\Delta^{op}})$ the category of $S^1$-symmetric spectra) of Morel-Voevodsky (see \cite{voevodsky-1999} for the unstable version). Here we generalize to the general case $\sh$ defined by Ayoub in \cite{ayoub-thesis-2} with a view towards the version for diagrams and applications to specialization systems. Moreover, we study operations at the level of diagram of algebraic spaces which helps us generalize the notion of motivic nearby functors of algebraic spaces.

\subsection{Motives} We define $\sh$ for algebraic spaces and investigate some first basic properties. 

\begin{defn}
A \textit{Nisnevich atlas} (or a Nisnevich covering) of an algebraic space $X$ is a surjective \'etale morphism $f \colon U \longrightarrow X$ morphism and there exists a stratification of $X$ such that $f$ admits a section over each stratum. Equivalently, it is enough to require that field-valued points of $X$ can be lifted to $U$.
\end{defn}
We define three sites $(\Sm/X)_{\textnormal{sch}}, (\Sm/X)_{\textnormal{rep}},(\Sm/X)$ whose objects are smooth schemes over $X$, smooth, representable algebraic spaces over $X$, smooth algebraic spaces over $X$ and morphisms are morphisms over $X$. Let $\mfrak$ be either $\spect^{\Sigma}_{S^1}(\sets^{\Delta^{op}})$ the category of symmetric $S^1$-spectra or $\mathrm{Comp}(\Lambda)$ the category of complexes over a $\Lambda$ (commutative with unity). Both of these categories have a suitable theory of homotopy groups; in the former case, it is stable homotopy groups and in the latter case it is homology groups. Let $X$ be an algebraic space and let $\preshv_?(\Sm/X,\mfrak)$ be the category of presheaves on $(\Sm/X)_?$ with values in $\mfrak$. The category $\preshv_?(\Sm/X,\mfrak)$ can be endowed with the \textit{projective model structure} whose weak equivalences and fibrations are objectwise. The $\tau$-\textit{local equivalences} are those morphisms of presheaves in $\preshv_?(\Sm/X,\mfrak)$ inducing isomorphisms on associated sheaves of presheaves of homotopy groups and this gives us the category $\preshv_{\tau,?}(\Sm/X,\mfrak)$. 
\begin{prop} \label{equivalences of topoi}
    Let $X$ be an algebraic space and $\tau$ be either the \'etale topology $\et$ or the Nisnevich topology $\nis$, then the canonical inclusions 
\begin{equation*}
    (\Sm/X)_{\tau,\textnormal{sch}} \longrightarrow (\Sm/X)_{\tau,\textnormal{rep}}  \longrightarrow (\Sm/X)_{\tau}
\end{equation*}
induce Quillen equivalences
\begin{equation*}
    \preshv_{\tau,\textnormal{sch}}(\Sm/X,\mfrak) \simeq  \preshv_{\tau,\textnormal{rep}}(\Sm/X,\mfrak) \simeq \preshv_{\tau}(\Sm/X,\mfrak).
\end{equation*}
\end{prop}
\begin{proof}
First we claim that the inclusions
\begin{equation*}
    (\Sm/X)_{\tau,\textnormal{sch}} \longrightarrow (\Sm/X)_{\tau,\textnormal{rep}}  \longrightarrow (\Sm/X)_{\tau}
\end{equation*}
induce equivalence of topoi 
\begin{equation*}
    \shv_{\tau,\textnormal{sch}}(\Sm/X) \simeq  \shv_{\tau,\textnormal{rep}}(\Sm/X) \simeq \shv_{\tau}(\Sm/X).
\end{equation*}
Indeed, the proof is very similar to the proof of \cite[Lemma 66.18.3]{stacks-project} with slight modifications. We use \cite[Lemma 7.29.1]{stacks-project}. The morphism sending $U/X$ to $U/X$ is clearly continuous. For cocontinuity, let $U \in (\Sm/X)_{\nis,\textnormal{sch}}$ and $\left \{v_i \colon V_i \longrightarrow U \right \}_{i \in I}$ be a Nisnevich covering of $U$ considered as an object in $(\Sm/X)_{\tau}$. For each $i \in I$, we choose a Nisnevich covering $u_i \colon U_i \longrightarrow V_i$ (this is possible thanks to \cite[Theorem 6.4]{knutson-1971}) with $U_i$ a scheme, the family $\left \{v_i \circ u_i \colon U_i \longrightarrow X \right \}_{i \in I}$ refines the given family because a composition of Nisnevich covering is a Nisnevich covering. The given functor is fully faithful, so conditions $(3)$ and $(4)$ in \cite[Lemma 7.29.1]{stacks-project} are automatically satisfied. The same reason as for cocontinuity shows that condition $(5)$ holds true and we are done. Now we assert that the existence of Quillen equivalences 
\begin{equation*}
    \preshv_{\tau,\textnormal{sch}}(\Sm/X,\mfrak) \simeq  \preshv_{\tau,\textnormal{rep}}(\Sm/X,\mfrak) \simeq \preshv_{\tau}(\Sm/X,\mfrak).
\end{equation*}
is a special case of \cite[Proposition 4.4.56]{ayoub-thesis-2} with respect to the following choice: $\mathcal{S} = (\Sm/X)_{\textnormal{sch}}$, the Yoneda embedding $(\Sm/X)_{\textnormal{sch}} \longhookrightarrow \shv_{\tau,\textnormal{sch}}(\Sm/X) \simeq \shv_{\tau}(\Sm/X)$ has image contained in $\mathcal{T} = (\Sm/X)$. The topology $\tau$ on $\mathcal{T}$ is clearly the finest topology on which $F \colon (\Sm/X) \longrightarrow \sets$ are sheaves for $F \in \shv_{\tau}(\Sm/X)$. The same argument applies to $(\Sm/X)_{\textnormal{sch}}$ and $(\Sm/X)_{\textnormal{rep}}$ and this completes the proof.  
\end{proof}
From the point of homotopy theory, proposition \ref{equivalences of topoi} shows that when working with model structures of sheaves, we do not distinguished between using $(\Sm/X)_{\tau,\textnormal{sch}}, (\Sm/X)_{\tau,\textnormal{rep}}, (\Sm/X)_{\tau}$ and thus we can just use the largest one, namely $(\Sm/X)_{\tau}$. We perform Bousfield localization $\preshv_{\tau}(\Sm/X,\mfrak)$ at the class consisting of morphisms $\mathbb{A}_U^1 \otimes A \longrightarrow U \otimes A$ with $U \in \Sm/X, A \in \mfrak$. The resulting category is the $(\mathbb{A}^1,\tau)$-\textit{local model structure} $\preshv_{(\mathbb{A}^1,\tau)}(\Sm/X,\mfrak)$. Let $T_X$ denote the presheaf 
\begin{equation*}
    T_X = \frac{\gm_{m,X} \otimes \mathds{1}}{X \otimes \mathds{1}}
\end{equation*}
We consider the category $\spect^{\Sigma}_{T_X}(\preshv_{(\mathbb{A}^1,\tau)}(\Sm/X,\mfrak))$ of symmetric $T_X$-spectra. We endow this category with the $(\mathbb{A}^1,\tau)$-\textit{local stable projective model structure} and set
\begin{equation*}
    \sh(X) \coloneqq \mathbf{Ho}_{(\mathbb{A}^1,\tau)-st}\big( \spect^{\Sigma}_{T_X}(\preshv_{(\mathbb{A}^1,\tau)}(\Sm/X,\mfrak)) \big).
\end{equation*}
If $f \colon X \longrightarrow Y$ is a morphism of algebraic $S$-spaces. By taking derived functors, we get suitable adjunctions of exact functors $(f^* \dashv f_*) \colon \sh(Y) \longrightarrow \sh(X)$ and $(f_{\#} \dashv f^*) \colon \sh(X) \longrightarrow \sh(Y)$ provided that $f$ is smooth. 
\begin{prop} \label{descent of motives}
Given a $\tau$-covering $\left \{u_i \colon U_i \longrightarrow X \right \}_{i \in I}$ of algebraic spaces, the functor 
    \begin{equation*}
        \prod_{i\in I} (u_i)^* \colon  \sh(X)  \longrightarrow  \prod_{i \in I}\sh(U_i) 
    \end{equation*}
is conservative in the following cases:
\begin{enumerate}
    \item $\tau_1 = \nis$ and any $\mfrak,\tau$. 
    \item $\tau_1 = \tau $ and any $\mfrak$.
\end{enumerate}
\end{prop}

\begin{proof}
The first case follows from the localization property (see for instance, \cite[Corollaire 1.4.4]{ayoub-thesis-1}) and the second case is similar to \cite[Lemme 3.4]{ayoub-realisation-etale}. 
\end{proof}
\begin{cor} \label{compactly generated property of alg spaces}
  Let $u \colon U \longrightarrow X$ be a $\tau$-atlas. Assume that $\sh(U)$ is compactly generated by a set $\mathcal{S}$ of compact objects, then $\sh(X)$ is compactly generated by the set $u_{\#}(\mathcal{S})$. 
\end{cor}

\begin{proof}
Let $A$ be a compact object in $\sh(U)$. Let us show first that $u_{\#}(A)$ is a compact object in $\sh(X)$. By definition, we have to show that
\begin{equation*}
    \Hom \big(u_{\#}(A),\bigoplus_{i \in I}B_i \big) \longrightarrow \bigoplus_{i \in I}\Hom \big(u_{\#}(A),B_i \big)
\end{equation*}
is bijective for any family of objects $(B_i)_{i \in I}$ in $\sh(X)$. By adjunction, the morphism above corresponds to 
\begin{equation*}
    \Hom \bigg(A,u^*\big(\bigoplus_{i \in I}B_i\big) \bigg) \longrightarrow \bigoplus_{i \in I}\Hom \big(A,u^*(B_i) \big).
\end{equation*}
Now this follows from the fact that $A$ is compact and $u^*$, being a left adjoint, commutes with colimits. To finish, we have to show that the set $u_{\#}(S)$ detects zero object, but this follows from the fact that $u^*$ is conservative (proposition \ref{descent of motives}) and $S$ is a generating set of $\sh(U)$. 
\end{proof}
\subsection{Four operations for algebraic spaces}
The proposition above allows us to reduce many results to known results for schemes. 
\begin{prop}
  The following properties of $\sh(X)$ hold: 
  \begin{enumerate}
      \item $\sh(\varnothing) = 0$. 
      \item If $i$ is an immersion, then $i^*i_* \longrightarrow \id$ is an isomorphism .
      \item Given a cartesian square
      \begin{equation*}
      \begin{tikzcd}[sep=large]
          X' \arrow[r,"g'"] \arrow[d,"f'",swap] & X \arrow[d,"f"] \\ 
          Y' \arrow[r,"g"]  & Y
      \end{tikzcd}
      \end{equation*}
      with $g$ smooth, then the natural base change morphism 
      \begin{equation*}
          \Ex_{\#}^* \colon (f')_{\#}(g')^* \longrightarrow g^*f_{\#}
      \end{equation*}
      is an isomorphism. 
      \item Let $j$ be an open immersion and $i$ be its closed complement, then $(j^*,i^*)$ is conservative. Consequently, there are distinguished triangles 
        \begin{equation*}
            \begin{split}
                j_{\#}j^* \longrightarrow \id \longrightarrow i_*i^* \longrightarrow +1 \\ 
                i_*i^! \longrightarrow \id \longrightarrow j_*j^* \longrightarrow +1. 
            \end{split}
        \end{equation*}
        \item Let $p \colon \mathbb{A}^1_X \longrightarrow X$ be the projection, then $\id \overset{\sim}{\longrightarrow} p_*p^*$ is an isomorphism. 
        \item Let $f \colon V \longrightarrow X$ be a smooth morphism (not necessarily representable) and $s \colon X \longrightarrow V$ be a section, then $\th(f,s) \coloneqq f_{\#}s_* \colon \sh(X) \longrightarrow \sh(X)$ is an equivalence with a quasi-inverse given by $\th^{-1}(f,s) = s^!f^*$.
  \end{enumerate}
\end{prop} 

\begin{proof}
Let us check only the locality axiom. Let $u \colon W \longrightarrow X$ be a Nisnevich atlas of $X$. We obtain cartesian squares by base change
\begin{equation*}
    \begin{tikzcd}[sep=large]
        W_U \arrow[r,hook,"j_W"] \arrow[d,"u_U",swap]  & W \arrow[d,"u"] & W_Z \arrow[d,"u_Z"] \arrow[l,"i_W",swap,hook] \\ 
        U \arrow[r,"j"] & X & Z \arrow[l,"i",swap] 
    \end{tikzcd}
\end{equation*}
then it is easy to see that $(j^*,i^*)$ is conservative if $((j_W)^*,(i_W)^*)$ is conservative but the latter is true since $W$ is a scheme. 
\end{proof}

With the result above, we can construct a formalism of four operations for morphisms of algebraic spaces. 
\begin{prop} \label{proper + smooth base change theorem}
  Given a cartesian square of morphisms of algebraic spaces
  \begin{equation*}
      \begin{tikzcd}[sep=large]
          X' \arrow[r,"g'"] \arrow[d,"f'",swap] & X \arrow[d,"f"] \\ 
          Y' \arrow[r,"g"]  & Y,
      \end{tikzcd}
  \end{equation*}
 then the natural transformation
 \begin{equation*} 
     \Ex^*_* = \Ex^*_*(f,g) \colon g^*f_* \longrightarrow (f')_*(g')^*
 \end{equation*} 
 is an isomorphism if either $g$ is smooth or $f$ is proper (not necessarily representable). 
\end{prop}
\begin{proof}
When $g$ is smooth, the isomorphism is obtained by simply taking the adjoint of the isomorphism $f^*g_{\#} \simeq (g')_{\#}(f')^*$. Let us verify second case when $f$ is proper, representable. As remarked at conventions, there exists a $\tau$-covering $v \colon V \longrightarrow Y$ with $V$ being a scheme (in particular, $v$ is smooth since it is at least \'etale). By base change, we obtain a cube
   \begin{equation*}
       \begin{tikzcd}[sep=large]
       & U' \arrow[dl,"u'",swap] \arrow[rr,"k'"] \arrow[dd,"h'",pos=0.7,swap,dashed] & &U \arrow[dl,"u",swap] \arrow[dd,"h",pos=0.6] \\
       X' \arrow[rr, crossing over,"g'",pos=0.7] \arrow[dd,"f'",pos=0.4,swap] & & X    \\
        &V'  \arrow[dl,"v'",swap,dashed] \arrow[rr,"k",pos=0.4,dashed] & &V \arrow[dl,"v"] \\
          Y' \arrow[rr,"g"] & &Y  \arrow[from=uu, crossing over,"f",pos=0.4]
        \end{tikzcd}
    \end{equation*}
in which every face is cartesian. By proposition \ref{descent of motives}, we only need to check that $(v')^*g^*f_* \longrightarrow (v')^*(f')_*(g')^*$ is an isomorphism. We note that $(v')^*(f')_* \simeq (h')_*(u')^*$ by the smooth base change theorem, hence by transitivity of the base change structure, we reduce to checking $k^*h_* \simeq (h')_*(k')^*$, which is true by the case of schemes. To verify the general case, we use the Chow lemma (see \cite[Chapter Four, Theorem 3.1]{knutson-1971}) together with noetherian induction (see \cite[Chapter Two, Proposition 5.18]{knutson-1971}). Suppose that $\Ex^*_*(g,f \circ i)$ is an isomorphism for any closed immersion $i \colon Z \longhookrightarrow X$ with $Z_{\textnormal{red}} \neq X_{\textnormal{red}}$. By the Chow lemma, there exists a projective, birational morphism $p \colon Z \longrightarrow X$ such that $f \circ p$ is projective. We form a cartesian square 
\begin{equation*}
    \begin{tikzcd}[sep=large]
        T \arrow[r,"t"] \arrow[d,"q",swap] & Z \arrow[d,"p"] \\ 
        U \arrow[r,"j"] & X 
    \end{tikzcd}
\end{equation*}
where $j \colon U \longhookrightarrow X$ is a dense open immersion over which $p$ is an isomorphism, i.e., $q$ is an isomorphism. We let $i \colon Z \longhookrightarrow X$ denote the closed complement of $j$ endowed with the reduced structure. The morphism $g^*f_* \longrightarrow (f')_*(g')^*$ is an isomorphism if and only if both
\begin{align*}
   \Ex_*^*(g,f \circ i) \colon  g^*f_*i_* & \longrightarrow (f')_*(g')^*i_* \\ 
   \Ex_*^*(g,f \circ j) \colon  g^*f_*j_* & \longrightarrow (f')_*(g')^*j_*
\end{align*}
are isomorphisms. By the induction hypothesis, the first morphism is an isomorphism. For the second morphism, we note that $j_* = p_*t_*q^*$ (since $q$ is an isomorphism) and hence the results follows from the projective case ($f \circ p$ projective).
\end{proof}
Since a morphism $f \colon X \longrightarrow Y$ of algebraic $S$-spaces is compactifiable (see \cite{conrad-2012}), we can follow the classical formulation of
\begin{cor}
  Let $f \colon X \longrightarrow Y$ be a morphism of algebraic $S$-spaces, then $f_!$ admits a right adjoint $f^!$. 
\end{cor}

\begin{proof}
By the Nagata compactification theorem, we can treat two separated cases: $f$ is an open immersion and $f$ is proper. In the former case, $f_! = f_{\#}$ has $f^*$ as its right adjoint. Let us prove the latter case. Thanks to the Brown representability theorem and proposition \ref{compactly generated property of alg spaces}, we only have to show that $f_*$ preserves direct sums if $f$ is proper. As in the proof of the proper base change theorem above, we can reduce to the case when $f$ is projective. In fact, we can prove the result for any representable morphism. We can choose $\tau$-atlases $V \longrightarrow Y$ and $U \longrightarrow X$ fitting into a cartesian square
\begin{equation*}
    \begin{tikzcd}[sep=large]
        U \arrow[d,"u",swap] \arrow[r,"f'"] & V \arrow[d,"v"] \\ 
        X \arrow[r,"f"] & Y.
    \end{tikzcd}
\end{equation*}
Since $u^*$ is conservative and it commutes with $f_*$ and direct sums, we only have to check that $(f')_*$ commutes with direct sums. This is true for schemes, see for instance \cite[Corollary 2.3.13]{cisinski+deglise}.
\end{proof}

Since we only use base change properties rather than all properties, we leave it as an exercise to the reader to check that a full-fledged version of four-functor formalism for $\sh$ (like the case of $\mathbf{SH}$ as established in \cite[Theorem 5.5.1]{chowdhury-2024}) is available. 

\subsection{Four operations for diagrams of algebraic spaces}
Before constructing operations at the level of diagrams of algebraic $S$-spaces, we need some definitions and auxiliary results. We begin with some basic definitions.
\begin{defn}
\begin{enumerate}
    \item A \textit{diagram of $S$-algebraic spaces} $(\Fscr,I)$ consists of a small category $I$ and a functor $\Fscr \colon I \longrightarrow \algspc_S$. 
    \item  A \textit{morphism} $(f,\alpha) \colon (\Gscr,J) \longrightarrow (\Fscr,I)$ consists of a functor $\alpha \colon J \longrightarrow I$ and a natural transformation $\Gscr \longrightarrow \Fscr \circ \alpha$.  
    \item Let $\mathbf{P}$ be a property of morphisms of algebraic spaces, then a morphism $(f,\alpha) \colon (\Gscr,J) \longrightarrow (\Fscr,I)$ is said to have $\mathbf{P}$ if for each $j \in J$, the morphism $f(j) \colon \Gscr(j) \longrightarrow \Fscr(\alpha(j))$ has $\mathbf{P}$.
    \item Let $(\Fscr,I)$ be a diagram of $S$-algebraic spaces, the category $\Sm/(\Fscr,I)$ has objects as $(U,i)$ with $i \in I$ an object and $U \longrightarrow \Fscr(i)$ an object in $(\Sm/\Fscr(i))$.
    \item Let $\tau$ denote either the \'etale topology $\et$ or the Nisnevich topology $\nis$. The category $\Sm/(\Fscr,I)$ is endowed with the $\tau$-topology whose a covering of $(U,i)$ is simply a $\tau$-covering of $U$.
    \end{enumerate}
\end{defn} 
Let $\preshv(\Sm/(\Fscr,I),\mfrak)$ be the category of presheaves on $\Sm/(\Fscr,I)$ with values in $\mfrak$. The category $\preshv(\Sm/(\Fscr,I),\mfrak)$ admits a semi-projective model structure whose weak equivalences are objectwise weak equivalences and cofibrations are those becoming surjections in all $\preshv(\Sm/\Fscr(i),\mfrak)$ with $i \in I$. The $\tau$-local equivalences are those morphisms of presheaves in $\preshv(\Sm/(\Fscr,I),\mfrak)$ inducing isomorphisms on associated sheaves of presheaves of homotopy groups. We can define the category $\shbb(\Fscr,I)$ by running through usual steps like the case of $\sh(X)$ and there are adjunctions $((f,\alpha)^* \dashv (f,\alpha)_*) \colon \shbb(\Fscr,I) \longrightarrow  \shbb(\Gscr,J) $ and $((f,\alpha)_{\#} \dashv (f,\alpha)^*) \colon \shbb(\Gscr,J) \longrightarrow \shbb(\Fscr,I) $ (provided that $(f,\alpha)$ is smooth objectwise in the latter case). They satisfy axioms \textbf{AlgDer 0} to \textbf{AlgDer 5} of an algebraic derivator in the sense of \cite[Definition 2.4.12]{ayoub-thesis-1} with proofs same as the case of schemes (see the fourth chapter of the same reference).
\begin{lem} \label{conservativity of pullbacks of schemes in a diagram}
    Let $(\Fscr,I)$ be a diagram of algebraic $S$-spaces, the family of functors $i^* \colon \shbb(\Fscr,I) \longrightarrow \sh(\Fscr(i))$ is conservative. 
\end{lem}

\begin{proof}
The proof is identical to the proof of \cite[Lemme 4.5.25]{ayoub-thesis-2}.
\end{proof}
The following proposition, which is an analogue of \cite{bang-2024}, based on proposition \ref{proper + smooth base change theorem}

\begin{prop} \label{proper + smooth base change theorem for diagrams}
  Given a cartesian square of $S$-varieties
    \begin{equation*}
        \begin{tikzcd}[sep=large]
                  (\mathscr{G}',I) \arrow[r,"(g'{,}\beta)"] \arrow[rd,phantom,"\square"] \arrow[d,"f'",swap] & (\mathscr{G},I) \arrow[d,"f"] \\ 
         (\mathscr{F}',I) \arrow[r,"(g{,}\beta)"] & (\mathscr{F},I)
        \end{tikzcd}
    \end{equation*}
such that one of the following conditions is satisfied:
\begin{itemize}
    \item $f$ is proper objectwise;
    \item $(g,\beta)$ is smooth objectwise,
\end{itemize}
then the canonical base change morphism 
\begin{equation*}
    \Ex^*_* \colon (g,\beta)^*f_* \longrightarrow (f')_*(g',\beta)^*
\end{equation*}
is an isomorphism. 
\end{prop}

\begin{defn}
    Let $f \colon (\Gscr,I) \longrightarrow (\Fscr,I)$ be a morphism of diagrams of algebraic $S$-spaces. We say that $f$ is \textit{compactifiable} if there exists a factorization, called a \textit{compactification}
    \begin{equation*}
        \begin{tikzcd}[sep=large]
            (\mathscr{G},I) \arrow[r,"j"] \arrow[dr,"f",swap] & (\overline{\mathscr{G}},I) \arrow[d,"p"] \\ 
            & (\mathscr{F},I)
        \end{tikzcd}
    \end{equation*}
   with $p$ proper objectwise and $j$ cartesian, open objectwise, we define a functor
    \begin{equation*}
        f_! \coloneqq p_*j_{\#} \colon \shbb(\Gscr,I) \longrightarrow \shbb(\Fscr,I)
    \end{equation*}
    called a \textit{exceptional direct image} of $f$, which depends on a choice of the compactification.  There is an obvious natural transformation $f_! \longrightarrow f_*$. If $p,j$ can be chosen to be cartesian, we say that $f$ is \textit{strongly compactifiable} and such a factorization is said to be a \textit{strong compactification}.
\end{defn}

\begin{ex}
Any morphism $f \colon (\Gscr,I) \longrightarrow (\Fscr,I)$ obtained by base change of a morphism of algebraic $S$-spaces has a strong compactification thanks to the Nagata compactification theorem \cite{conrad-2012}. We call such a morphism \textit{scheme-like representation} and such a compactification a \textit{scheme-like compactification}. 
\end{ex}

\begin{prop} \label{support property of derivators}
Given a commutative (not necessarily cartesian) square
    \begin{equation*}
    \begin{tikzcd}[sep=large]
        (\mathscr{U}',I) \arrow[r,"u'"] \arrow[dr,phantom,"\textnormal{(C)}"] \arrow[d,"p'",swap] & (\mathscr{U},I) \arrow[d,"p"] \\ 
        (\mathscr{F}',I) \arrow[r,"u"] & (\mathscr{F},I).
    \end{tikzcd}
\end{equation*}
of diagrams of algebraic $S$-spaces where $u,u'$ are open objectwise and $p,p'$ are proper objectwise. If $u,u'$ are cartesian, then the canonical base change morphism 
\begin{equation*}
    \Ex_{\#*} \colon u_{\#}(p')_* \longrightarrow p_*(u')_{\#}
\end{equation*}
is an isomorphism.
\end{prop}
\begin{proof}
  First, we see that if $I = \mathbf{e}$, then this proposition can be deduce from the proper base change theorem \ref{proper + smooth base change theorem} (for schemes, this is called the \textit{support property} in \cite{cisinski+deglise}). Now we assume that $\textnormal{(C)}$ is cartesian then $u'$ is also cartesian.  By lemma \ref{conservativity of pullbacks of schemes in a diagram}, it is sufficient to check that for each $i \in I$, the morphism 
    \begin{equation*}
        i^*u_{\#}(p')_* \longrightarrow i^*p_*(u')_{\#}
    \end{equation*}
    is an isomorphism. Indeed, we see that
    \begin{align*}
        i^*u_{\#}(p')_* & \simeq u(i)_{\#}i^*(p')_* & \textbf{AlgDer 3d} \\ 
                    & \simeq u(i)_{\#}(p'(i))_*i^* & \textbf{AlgDer 3d} \\ 
                    & \simeq p(i)_*(u'(i))_{\#}i^* & \text{by the support property} \\ 
                    & \simeq p(i)_*i^*(u')_{\#}  \simeq i^*p_*(u')_{\#}.
    \end{align*}
    In general, we form the following commutative diagram 
    \begin{equation*}
    \begin{tikzcd}[sep=large]
    (\mathscr{U}',I) \arrow[dr,"\Phi"] \arrow[rrd,"u'",bend left = 15] \arrow[ddr,"p'",swap,bend right = 15] & & \\ 
         & (\mathscr{G},I) \arrow[r,"v"] \arrow[d,"q"] \arrow[rd,phantom,"\textnormal{(C')}"]  & (\mathscr{U},I) \arrow[d,"p"] \\ 
        & (\mathscr{F}',I) \arrow[r,"u"] & (\mathscr{F},I),
    \end{tikzcd}
    \end{equation*}
    in which $\textnormal{(C')}$ is cartesian. In this diagram, $\Phi$ is cartesian (because both $v,u'$ are cartesian), closed and open objectwise. We claim that $\Phi_{\#} \longrightarrow \Phi_*$ is an isomorphism. By lemma \ref{conservativity of pullbacks of schemes in a diagram} again, it is sufficient to check that $i^*\Phi_{\#} \longrightarrow i^*\Phi_*$ is an isomorphism for every $i^* \colon \shbb(\Gscr,I) \longrightarrow \sh(\Gscr(i))$, but this is a consequence of \textbf{AlgDer3d}. The morphism
    \begin{equation*}
        u_{\#}(p')_* = u_{\#}q_*\Phi_* \overset{\sim}{\longrightarrow} p_*v_{\#}\Phi_* \overset{\sim}{\longrightarrow} p_*v_{\#}\Phi_{\#}= p_*(u')_{\#}
    \end{equation*}
    is an isomorphism as desired. 
\end{proof}

Once we fix a choice of a scheme-like compactification, the operation $f_!$ does not depend on the choice of scheme-like compactifications and the connection $f_!g_! \simeq (f \circ g)_!$ can be performed in cases that $f,g$ are base change of two morphisms of algebraic spaces. We leave this detail to the reader since it is similar to schemes. We come to the right adjoint $f^!$ of $f_!$.

\begin{prop}
    Let $(\Fscr,I)$ be a diagram of algebraic $S$-spaces. For each compact object $A$ of $\shbb(\Fscr(i))$, the object $i_{\#}(A)$ is compact in $\shbb(\Fscr,I)$. Moreover, we know that
    \begin{equation*}
        \shbb(\Fscr,I) = \left <\left< i_{\#}A \mid i \colon (\Fscr(i),\mathbf{e}) \longrightarrow (\Fscr,I), A \in \shbbct(\Fscr(i)) \right> \right >.
    \end{equation*}
    In particular, $\shbb(\Fscr,I)$ is compactly generated. 
\end{prop}
\begin{proof}
    Let us prove that each $i_{\#}A$ is a compact object. By definition, we have to show that 
    \begin{equation*}
        \bigoplus_{j \in J}\Hom(i_{\#}A,B_j) \longrightarrow \Hom \left(i_{\#}A,\bigoplus_{j \in J}B_j \right)
    \end{equation*}
    is bijective for every family $(B_j)_{j\in J}$ of objects in $\shbb(\Fscr,I)$. By adjunction, it is equivalent to prove that
    \begin{equation*}
        \bigoplus_{j \in J}\Hom(A,i^*B_j) \longrightarrow \Hom \bigg(A,i^*\bigg(\bigoplus_{j \in J}B_j \bigg) \bigg)
    \end{equation*}
    is bijective. This is true since $i^*$ commutes with colimits and $A$ is a compact object. By \cite[Proposition 2.1.27]{ayoub-thesis-1}, it remains to prove that the family $\Hom(i_{\#}A[n],-)$ with $i \in I$, $A \in \shbbct(\Fscr(i))$, $n\in \mathbb{Z}$ is conservative. This is a consequence of the fact that for each $i$, the family $\Hom(A[n],-)$ is conservative and the family $i^*\colon \shbb(\Fscr,I) \longrightarrow \shbb(\Fscr(i))$ itself is conservative.
\end{proof}

\begin{prop} \label{right adjoint of morphism proper objectwise}
    Let $I$ be a small category and $p \colon (\Gscr,I) \longrightarrow (\Fscr,I)$ be a morphism proper objectwise between diagrams of algebraic $S$-spaces, then $p_*$ preserves direct sums. Consequently, $p_*$ admits a right adjoint $p^!$. 
\end{prop}

\begin{proof}
     Let $(B_j)_{j \in J}$ be a set of objects in $\shbb(\Gscr,I)$, we have to prove that the canonical morphism
    \begin{equation*} 
    \bigoplus_{j \in J}p_*B_j \longrightarrow p_*\left(\bigoplus_{j \in J}B_j \right)
    \end{equation*} 
    is an isomorphism. Since the family $i^* \colon \shbb(\Fscr,I) \longrightarrow \shbb(\Fscr(i))$ with $i \in I$ is conservative, we have to check that 
    \begin{equation*} 
    i^*\left(\bigoplus_{j \in J}p_*B_j \right) \longrightarrow i^*p_*\left(\bigoplus_{j \in J}B_j \right)
    \end{equation*} 
    is an isomorphism for every $i \in I$. This follows from the fact that $i^*$ commutes with colimits, the isomorphism $i^*p_* \simeq p(i)_*i^*$ and the case of schemes. By using a consequence of the Brown representability theorem (see for instance \cite[Corollaire 2.1.22]{ayoub-thesis-1}) and the preceding proposition, we see that $p_*$ admits a right adjoint. 
\end{proof}

\begin{cor}
    Let $f\colon (\mathscr{G},I) \longrightarrow (\mathscr{F},I)$ be a strongly compactifiable morphism, then $f_!$ admits a right adjoint $f^!$.
\end{cor}

\begin{proof}
    Since $f$ is compactifiable, we may treat two separated cases when $f$ is proper and open. If $f$ is proper, then $f_! = f_*$ admits an adjoint by proposition \ref{right adjoint of morphism proper objectwise}. If $f$ is open, then $f_! = f_{\#}$ has $f^*$ as its right adjoint. 
\end{proof}

\begin{prop} \label{proper base change theorem for diagrams}
  Given a cartesian square of diagrams of $S$-varieties
    \begin{equation*}
        \begin{tikzcd}[sep=large]
                  (\mathscr{G}',J) \arrow[r,"(g'{,}\beta)"] \arrow[rd,phantom,"\square"] \arrow[d,"f'",swap] & (\mathscr{G},I) \arrow[d,"f"] \\ 
         (\mathscr{F}',J) \arrow[r,"(g{,}\beta)"] & (\mathscr{F},I)
        \end{tikzcd}
    \end{equation*}
with $f$ compactifiable, then the canonical base change morphisms 
\begin{align*}
    \Ex^*_! \colon (g,\beta)^*f_! \longrightarrow (f')_!(g',\beta)^* \\ 
     \Ex^!_* \colon  f^!(g,\beta)_*  \longrightarrow (g',\beta)_*(f')^!
\end{align*}
are isomorphisms. By adjunction, one obtains 
\begin{align*}
    \Ex^{!*}\colon (g',\beta)^*f^! \longrightarrow (f')^!(g,\beta)^* \\ 
    \Ex^!_* \colon (g',\beta)_*(f')^! \longrightarrow f^!(g,\beta)_*
\end{align*}
\end{prop}

We end this appendix by a useful lemma 
\begin{lem} \label{direct images commute with homotopy colimits} 
    Let $f \colon X \longrightarrow Y$ be a morphism of algebraic $S$-spaces and $I$ be a small category, from the cartesian square
    \begin{equation*}
        \begin{tikzcd}[sep=large]
                  (X,I) \arrow[r,"p_I"] \arrow[d,"f",swap] & X \arrow[d,"f"] \\ 
                  (Y,I) \arrow[r,"p_I"] & Y
        \end{tikzcd}
    \end{equation*}
    we deduce an isomorphism 
    \begin{equation*}
        (p_I)_{\#}f_! \longrightarrow f_!(p_I)_{\#}. 
    \end{equation*}
\end{lem}

\begin{proof}
    If $X,Y$ are schemes, this is \cite[Proposition 1.2]{bang-2024}. In general, one can choose $\tau$-atlases and use \cite[Corollaire 2.4.24]{ayoub-thesis-1} and proposition \ref{proper base change theorem for diagrams} to reduce to the case of schemes. We leave the details to the readers. 
\end{proof}
\bibliographystyle{alpha}
\bibliography{ref}

\end{document}